\documentclass[a4paper,english]{jnsao}
\usepackage[utf8]{inputenc}
\usepackage[T1]{fontenc}
\usepackage{csquotes}
\usepackage{babel}
\usepackage[ruled,vlined]{algorithm2e}
\usepackage{tikz,pgfplots,pgfplotstable}
\usetikzlibrary{shapes.geometric}
\usetikzlibrary{fpu}
\usepgfplotslibrary{colorbrewer,fillbetween}
\usepackage{enumitem}\setlist{nosep}
\usepackage{mathtools}
\usepackage{comment}
\usepackage[indLines=false,noEnd=false]{algpseudocodex}
\usepackage[nameinlink,capitalize]{cleveref}
\usepackage{float}
\usepackage[noadjust]{cite}

\newcommand{\term}{\emph}

\newcommand{\field}[1]{\mathbb{#1}}
\newcommand{\N}{\mathbb{N}}

\newcommand{\R}{\field{R}}
\newcommand{\extR}{\overline \R}

\newcommand{\norm}[1]{\|#1\|}
\newcommand{\adaptnorm}[1]{\left\|#1\right\|}
\newcommand{\bignorm}[1]{\bigl\|#1\bigr\|}

\newcommand{\abs}[1]{|#1|}
\newcommand{\adaptabs}[1]{\left|#1\right|}

\newcommand{\grad}{\nabla}
\newcommand{\freevar}{\,\boldsymbol\cdot\,}

\newcommand{\Union}\bigcup
\newcommand{\Isect}\bigcap
\newcommand{\union}\cup
\newcommand{\isect}\cap
\newcommand{\bigunion}\bigcup
\newcommand{\bigisect}\bigcap

\newcommand{\defeq}{:=}

\newcommand{\downto}{\searrow}
\newcommand{\upto}{\nearrow}
\newcommand{\subdiff}{\partial}

\DeclareMathOperator*{\argmin}{arg\,min}

\DeclareMathOperator{\diam}{diam}

\def \uminusSym{\setbox0=\hbox{$\cup$}\rlap{\hbox
        to\wd0{\hss\raise0.5ex\hbox{$\scriptscriptstyle{-}$}\hss}}\box0}

\newcommand{\iprod}[2]{\langle #1,#2\rangle}
\newcommand{\adaptiprod}[2]{\left\langle #1,#2\right\rangle}
\def\llangle{\langle\kern-3pt\langle}
\def\rrangle{\rangle\kern-3pt\rangle}

\def \weaktostarSym{\setbox0=\hbox{$\rightharpoonup$}\rlap{\hbox
        to\wd0{\hss\raise1ex\hbox{$\scriptscriptstyle{*\,}$}\hss}}\box0}

\def\linear{\mathbb{L}}

\def\extR{\overline \R}

\def\this#1{#1^n}
\def\nexxt#1{#1^{n+1}}

\def\thisx{\this{x}}

\def\thisv{\this{v}}

\DeclareMathOperator{\prox}{prox}

\def\d{\,d}

\def\dualprod#1#2{\langle #1|#2\rangle}

\newcommand{\Meas}{\mathscr{M}}

\let\phi=\varphi
\let\epsilon=\varepsilon

\def\ifempty#1{\def\temp{#1}\ifx\temp\empty}
\def\Masses{\Meas(\Omega)}
\def\Plans{\Meas_+(\Omega^2)}

\def\dd{\,\mathrm{d}}

\def\Predual{C_0(\Omega)}

\newcommand{\diffwrt}[2]{#1^{(#2)}}

\def\X{X}

\def\triple{x}
\def\triplealt{\tilde x}

\def\dx{\dd\spatial}
\def\spatial{\xi}
\def\newspatial{\zeta}

\def\obs{\mathbf{A}}

\def\spaceofdomain{\R^d}

\def\julipu{L_{\diffwrt{J}{u}|u}}
\def\julipx{L_{\diffwrt{J}{u}|x}}
\def\jxlipu{L_{\diffwrt{J}{x}|u}}
\def\jxlipx{L_{\diffwrt{J}{x}|x}}
\def\eulipu{L_{\diffwrt{e}{u}|u}}
\def\eulipx{L_{\diffwrt{e}{u}|x}}
\def\jfullulipx{L_{\diffwrt{J}{u} \circ S}}

\def\Raux{R_{\mathrm{aux}}}

\renewrobustcmd{\downto}{{{\mathchoice%
            {\rotatebox[origin=c]{-20}{$\to$}}% display
            {\rotatebox[origin=c]{-20}{$\to$}}% text
            {\rotatebox[origin=c]{-20}{\scalebox{0.75}{$\to$}}}% subscript
            {\rotatebox[origin=c]{-20}{\scalebox{0.6}{$\to$}}}% subsubscript
}}}

\renewrobustcmd{\upto}{{{\mathchoice%
            {\rotatebox[origin=c]{20}{$\to$}}% display
            {\rotatebox[origin=c]{20}{$\to$}}% text
            {\rotatebox[origin=c]{20}{\scalebox{0.75}{$\to$}}}% subscript
            {\rotatebox[origin=c]{20}{\scalebox{0.6}{$\to$}}}% subsubscript
}}}

\theoremstyle{definition}
\newtheorem{assumption}[theorem]{Assumption}
\crefname{assumption}{Assumption}{Assumptions}
\crefrangeformat{equation}{\upshape{(#3#1#4)--(#5#2#6)}}

\usepackage{mdframed}

\mdfdefinestyle{examplebg}{
    backgroundcolor=black!6,%
    hidealllines=true,%
    innertopmargin=.3em,%
    innerbottommargin=.7em,%
    innerleftmargin=.7em,%
    innerrightmargin=.7em,%
    footnoteinside=false,
}

\surroundwithmdframed[style=examplebg]{example}

\manuscriptcopyright{}
\manuscriptlicense{}%
\manuscripteprinttype{arxiv}
\manuscripteprint{2605.12095}

\author{
    Thi Tam Dang\thanks{%
        Department of Mathematics and Statistics, University of Helsinki,
        \email{tam.dang@iki.fi}, \orcid{0009-0005-1893-5351}
    }
    \and
    Tuomo Valkonen\thanks{%
        MODEMAT Research Center in Mathematical Modeling and Optimization
        \emph{and}
        Department of Mathematics and Statistics, University of Helsinki,
        \email{tuomo.valkonen@iki.fi}, \orcid{0000-0001-6683-3572}}
    }
\shortauthor{T.~Valkonen}

\acknowledgements{%
    This research has been supported by the Jane and Aatos Erkko Foundation.
}

\title{Leak localisation with a measure source convection--diffusion model}
\shorttitle{Leak localisation with a measure source}

\date{2026-05-12 (revised: 2026-10-09)}

\begin{document}

\maketitle

\begin{abstract}
    Taking a variational approach, we study the inverse problem of locating gas leaks from line-of-sight concentration measurements using a convection–diffusion model with the source term a Radon measure. By imposing sparsity-promoting regularisation on this measure, we recover point sources---identifying both their locations and intensities---rather than diffuse approximations.
    We jointly estimate the underlying physical convection (wind) and diffusion parameters.
    Our main theoretical contribution is the analysis of Lipschitz dependence of the forward solution of the convection--diffusion equation with respect to its parameters: the measure, and the convection and diffusion fields.
    Numerically, we employ a semi-grid-free optimisation approach for reconstructing the source measure. Our experiments demonstrate accurate localisation, highlighting the potential of the method for practical gas emission detection.
\end{abstract}

\section{Introduction}

Modern industrial facilities require reliable gas leak detection to ensure safety and meet regulatory standards. Accurate quantification of greenhouse gas emissions, particularly methane, is essential for effective climate mitigation~\cite{balcombe2018methane}. Gas emission detection (GED) uses various instruments: remote sensing from aircraft and satellites, point-wise systems such as eddy-covariance towers, and open-path spectrometers. Satellite-based systems are efficient for regional to global emission estimation but often lack accuracy at local scales and in distinguishing different source types (for example, agricultural versus industrial)~\cite{engram2020remote,mcnorton2018attribution}. Eddy-covariance towers provide high-frequency data but only at fixed locations~\cite{miller2013anthropogenic,baldocchi2014measuring}. In contrast, open-path laser dispersion spectroscopy (LDS) helps to bridge this gap by providing path-averaged concentration measurements along multiple beams, which enables rapid two-dimensional emission mapping that is robust under varying atmospheric conditions~\cite{rella2015measuring,eaghestani2014analysis,voss2024multi,alden2019singleblind,hirst2020methane,weidmann2022locating}. Because all these instruments measure gas concentrations rather than emission rates directly, their data must be combined with an inversion model to localise and quantify the underlying emission sources.

Recovering the full gas concentration field from sparse line-of-sight measurements is an ill-posed inverse problem governed by a parabolic convection-diffusion PDE \eqref{eq:parabolic:convection-diffusion}.
In this work, the source term in the PDE will be a measure, to facilitate modelling discrete sources of leaks, both their locations and release rates:
\begin{equation}
    \label{eq:intro:mu}
    \mu = \sum_{i=1}^N \beta_i \delta_{x_i},
\end{equation}
where $\beta_i$ is the release rate of the leak, and $x_i \in \Omega \subset \spaceofdomain$ its location.
We denote by $\delta_x$ the Dirac measure concentrated at $x$.
To reconstruct such a $\mu$, we consider optimisation problems of the form
\begin{equation}
    \label{sec:intro:problem}
    \min_{\mu,k,c} \frac{1}{2}\norm{Au-b}^2 + \alpha \norm{\mu}_{\Masses} + \delta_{\ge 0}(\mu) + \Raux(k, c),
\end{equation}
where the concentration $u: \Omega \times [0, T] \to \R$ arises as a solution of the convection--diffusion equation for the sources $\mu$, convection field $c$, and diffusion field $k$.
The observation operator $A$ performs line integration of the concentration from laser sources to mirrors over multiple time instants: our observations are over time, while the length $T$ of the time interval is---for simplicity---assumed to be short enough that $\mu$ is time-independent.
The Radon-norm regularisation term with parameter $\alpha>0$ encourages sparsity of solutions \cite{bredies2019sparsity}: that an optimal $\mu$ would have the form \eqref{eq:intro:mu}, instead of being diffuse.
The term $\delta_{\ge 0}$ (an indicator function, not to be confused with the Dirac $\delta_x$-measure)  enforces that there are only sources, no sinks. The convection and diffusion fields' measurements and regularisation are modelled by $\Raux$.

The dependence of $u$ on $\triple=(\mu,k,c)$ is nonlinear. Conventionally, if the dependence were linear, and we would not have the auxiliary parameters $k$ and $c$, this type of measure optimisation problems would be solved by a conditional gradient (Frank--Wolfe) method  \cite{brediespikkarainen2013inverse,denoyelle2019sliding,duval2017sparse,walter2019linear,blank2017extension,bredies2021linear} in a grid-free fashion. Semismooth Newton, particle-based, and semi-infinite approaches have also been developed  \cite{casas2012approximation,casas2013parabolic,chizat2021sparse,chizat2018global,chizat2023convergence,flinth2020linear}.
More recently \cite{tuomov-pointsource,tuomov-unbalanced}, we introduced more conventional forward-backward and primal-dual methods that more easily extend to nonconvex problems and auxiliary parameters.
We will exploit those in \cref{sec:numerical} to solve the problem \eqref{sec:intro:problem}.
The algorithms are described in more detail in \cref{sec:optimisation}.
To be able to apply these algorithms, we will require the PDE solution operator $S_u: (\mu, k, c) \to u$ to be Lipschitz-differentiable in appropriate spaces.
We will in \cref{sec:parabolic} perform the relevant analysis.
This forms the principal theoretical contribution of this manuscript.
In \cref{sec:implementation}, we discuss in detail our finite element approach and the numerical implementation, based on accurate ray-tracing, similar to what has been done in computed tomography~\cite{kak1998principles} and optical tomography~\cite{arridge1999optical}. Our approach computes the adjoint in a numerically consistent manner.

Several works have studied inverse-source problems for elliptic, parabolic, and wave-type equations from a fully PDE-based viewpoint. These works have developed methods for identifying the locations and strengths of point sources from suitable observations; see, for example, \cite{badia2000inverse}. For wave equations, \cite{ohe2019raeltime} develops an algebraic reciprocity-gap method for the real-time reconstruction of multiple moving point and dipole sources from boundary measurements. More recently, \cite{qiu2025traceability,qiu2025water} use dynamic complex geometrical optics solutions to identify time-dependent diffusion sources from boundary measurements and establish uniqueness and stability
results for source component. Furthermore, \cite{qiu2026optimal} prove optimal-transport stability for inverse point-source problems in elliptic and parabolic equations, controlling the distance between source measures by boundary-data discrepancies.

Our work, in contrast, takes a variational approach to the solution of the inverse problem.
It, therefore, focuses on the analysis of the forward problem and develops a well-posed solution operator with local Lipschitz continuity and Fréchet differentiability. These properties then provide the basis for a variationally regularised model and an adjoint-based optimization method for recovering the unknown sources.

\paragraph{Notation}

For normed spaces \(X\) and \(Y\), we write \(\linear(X;Y)\) for the space of bounded linear operators from \(X\) to \(Y\). The dual space of \(X\) is denoted by \(X^*\). For a functional \(F: X \to \R\), we denote by \(DF(x)\) the Gâteaux derivative of \(F\) at \(x\), and by \(F'(x)\in X^*\) the Fréchet derivative at \(x\), whenever it exists. If \(X\) is a Hilbert space, \(\grad F(x)\in X\) stands for the Riesz representative of \(F'(x)\), i.e., the gradient of \(F\) at \(x\). For partial derivatives, we use the notation $\diffwrt{F}{u}(u, x)$ when differentiating with respect to a specific named argument.

The space $\Masses$ is that of Radon measures on the domain $\Omega \subset \spaceofdomain, 1 \le d \le 3$.

We write $\dualprod{x^*}{x}_{X^*,X}=x^*x$ for the dual pairing, and $\iprod{x}{y}_X$ for the inner product in Hilbert spaces. When the spaces are clear from the context, we do not indicate them.

We generally write $\xi \in \R^d$ for the spatial variable, with $x$ reserved for the overall unknown.

We write $\delta_A$ for the $\{0,\infty\}$-valued indicator function of a set $A$.
We also use $\delta_x$ to denote the Dirac mass at the point $x$.

\section{The convection--diffusion equation with measure sources}
\label{sec:parabolic}

This section formulates the PDE-constrained inverse problem under study (\cref{sec:problem}).
The forward PDE \eqref{eq:parabolic:convection-diffusion} is a parabolic convection-diffusion problem with a measure source $\mu$, modelling the locations and intensities of the unknown leaks. Also the convection and diffusion coefficient are unknown.
We first establish well-posedness of the PDE (\cref{thm:existence}, \cref{sec:existence}), and Lipschitz continuity of its solution operator (\cref{sec:lipschitz}).
To differentiate the solution operator, \cref{sec:adjoint} derives the adjoint equation and corresponding Lipschitz estimates. Finally, \cref{sec:parabolic:dataterm} combines the forward and adjoint equations to obtain Lipschitz estimates for the differential of the data term.
This is necessary for the application of the algorithms of \cref{sec:optimisation}.

\subsection{Problem setting}
\label{sec:problem}

Let $\Omega \subset \spaceofdomain$, ($1 \le d \le 3$), be a bounded domain with convex geometry or with boundary $\Gamma$ of class $C^{1,1}$.
For a given end time $T>0$, write $I \defeq [0,T]$ and $\Omega_T \defeq \Omega \times (0,T)$.
Recalling \eqref{sec:intro:problem}, we consider the PDE-constrained optimisation problem
\begin{equation}
    \label{eq:problem}
    \min_{\mu, k, c}~
    J(S_u(\mu, k, c)) + R(\mu, k, c)
\end{equation}
for the convection--diffusion PDE solution operator $S_u$ (described below) and the specific choices
\begin{equation}
    \label{eq:problem:specific}
    J(u) = \frac{1}{2} \norm{A u - b }_2^2
    \quad\text{and}\quad
    R(\mu, k, c) = \alpha \norm{\mu}_{\Masses} + \delta_{\ge 0}(\mu) +  \Raux(k, c).
\end{equation}
The unknowns are the measure $\mu \in \Masses$, the diffusion coefficient $k \in L^\infty(\Omega_T)$, and the convection coefficient $c \in L^\infty(\Omega_T; \spaceofdomain)$.
Here $k_{\min} > 0$ ensures uniform ellipticity (non-degenerate diffusion), while $k_{\max}, c_{\max} > 0$ enforce physical bounds for numerical stability.
Physically, $k(\xi,t)>0$ controls diffusive spreading, $c(\xi,t)$ represents background advection (wind currents). The measure $\mu$ models the locations and intensities of gas leaks.
The Radon-norm regulariser, with regularisation parameter $\alpha>0$, seeks to impose sparsity of this measure: that it would consist of a finite number of point sources (Dirac masses).
The indicator function $\delta_{\ge 0}$ enforces that we only have sources, no sinks.
The convection and diffusion fields are regularised by the convex function $\Raux: L^\infty(\Omega_T) \times L^\infty(\Omega_T; \spaceofdomain) \to \extR$; we will in \cref{sec:numerical} take it as a simple quadratic regulariser combined with constraints that restrict $k$ and $c$ to scalar subspaces
\begin{align*}
    K & \subset \{ k \in L^\infty(\Omega_T) : k_{\min} \le k(\xi,t) \le k_{\max} \text{ a.e.}  \ \ \xi \in \Omega\}
    \quad\text{and}
    \\
    C & \subset L^\infty(\Omega_T; [-c_{\max}, c_{\max}]^d),
\end{align*}
The measurements  $b \in \R^m$ are obtained with a system of lasers and mirrors, modelled by the observation operator $A \in \linear(L^2(\Omega_T); \R^m)$.
For each $\triple=(\mu, k, c)$, the corresponding gas concentration $u=S_u(\triple) \in L^p(I; W^{1,p}(\Omega)) $ with $p \in [1,\infty)$ is the unique solution of the convection--diffusion equation with initial and boundary values,
\begin{equation}
    \label{eq:parabolic:convection-diffusion}
    \begin{aligned}
        \partial_t u - \grad \cdot (k \grad u)  \, + c \cdot \grad u & = \mu &  & \text{in } \Omega_T,
        \\
        u                                                            & = g   &  & \text{on } \Sigma_1 \defeq \Gamma_1 \times (0,T),
        \\
        \frac{\partial u}{\partial n}                                & = 0   &  & \text{on } \Sigma_2 \defeq \Gamma_2 \times (0,T),
        \\
        u(\cdot,0)                                                   & = u_0 &  & \text{in } \Omega,
    \end{aligned}
\end{equation}
where $\Gamma_1$ and $\Gamma_2$ are disjoint with $\Gamma_1 \union \Gamma_2 = \partial \Omega$;
$u_0 \in L^2(\Omega)$ denotes the given initial condition, and $g \in L^2(\Sigma_1)$ the prescribed boundary flux.

\subsection{Existence of weak solutions}
\label{sec:existence}

Let $I = [0,T]$ and $1 < p < \infty$. We denote by $W^{1,p}(\Omega)$ the Sobolev space of functions in $L^p(\Omega)$ whose distributional derivatives belong to $L^p(\Omega)$, and set $W^{-1,p'}(\Omega)$ to be its dual, where $1/p + 1/p' = 1$. These spaces are reflexive and separable. Consequently, the space
\[
    L^2(I; W^{1,p}(\Omega)) \defeq \{ u: [0,T] \to W^{1,p}(\Omega) \text{ measurable} \mid \norm{u}_{L^2(I; W^{1,p}(\Omega))} < \infty \},
\]
equipped with the norm
\[
    \norm{u}_{L^2(I; W^{1,p}(\Omega))} \defeq \left( \int_0^T \norm{u(t)}_{W^{1,p}(\Omega)}^2 \dd t \right)^{1/2},
\]
is a separable and reflexive Banach space, whose dual can be identified with $L^2(I; W^{-1,p'}(\Omega))$.

We also work with $L^2(I; C_0(\Omega))$ equipped with the norm
\[
    \norm{u}_{L^2(I; C_0(\Omega))} \defeq \left( \int_0^T \norm{u(t)}_\infty^2 \dd t \right)^{1/2} < \infty.
\]
The action of $\mu$ on time-dependent functions $v \in L^2(I;C_{0}(\Omega))$ is defined through the canonical duality between
$\Masses$ and $C_{0}(\Omega)$, integrated over the time interval $I$.
More precisely, we define
\[
    \iprod{\mu}{v}_{\Masses,\,L^2(I;C_{0}(\Omega))}
    \defeq \int_{0}^{T} \iprod{\mu}{v(t) }_{\Masses,\,C_{0}(\Omega)}
    \dd t,
\]
where
\[
    \iprod{\mu}{v }_{\Masses,\,C_{0}(\Omega)}
    \defeq \int_{\Omega} v \dd \mu,
    \qquad v \in C_{0}(\Omega).
\]

Set
\[
    W_{\Gamma_1}^{1,p}(\Omega) \defeq \{ \tilde{v} \in C^{\infty}(\widebar{\Omega}), \,  \tilde{v}|_{\Gamma_1} = 0 \}
    \quad\text{and}\quad
    W_{\Gamma_1}^{-1,p}(\Omega)
    \defeq
    \bigl[W_{\Gamma_1}^{1,p'}(\Omega)\bigr]^*.
\]
For simplicity of notation, we define the space of concentrations (i.e., the state space in
the terminology of optimal control)
\[
    U \defeq L^2\bigl(I; W^{1,p}(\Omega)\bigr),
\]
and the space of test functions and its dual,
\[
    W \defeq \bigl\{ v \in C^1(\overline{\Omega_T}) \,\bigm|\, v(\cdot,t) \in W^{1,p'}_{\Gamma_1}(\Omega) \text{ for all } t \in \overline I,\ v(\cdot,T)=0 \bigr\}
    \quad \text{and}\quad
    W^* \defeq L^2\bigl(I; W^{-1,p}_{\Gamma_1}(\Omega)\bigr),
\]
where $1 < p < d/(d-1)$ for $d \ge 2$ (any $1<p<\infty$ for $d=1$). Since $p' = p/(p-1) > d$, Morrey's embedding gives
\[
    W^{1,p'}(\Omega)\hookrightarrow C(\overline\Omega),
\]
so the pairing $\iprod{\mu}{v(t)}$ is well-defined for $\mu\in\Masses$.

We may now introduce the concept of a weak solution to \eqref{eq:parabolic:convection-diffusion}:

\begin{definition}[Weak Solution]
    \label{def:weak}
    Given $\triple=(\mu,k,c)$, $u_0\in L^2(\Omega)$, and $g\in L^2(\Sigma_1)$, let
    $u\in U$ satisfy
    $u-\tilde g\in L^p\! \left(I;W^{1,p}_{\Gamma_1}(\Omega)\right)$, where $\tilde g $ is a fixed extension of the Dirichlet data $g$. We call $u$  a weak solution to problem \eqref{eq:parabolic:convection-diffusion}
    if
    \[
        -\int_0^T\!\!\!\int_\Omega u\,\partial_t v \, \dd \xi \dd t
        -\int_\Omega u_0\, v(\cdot,0)\, \dd \xi
        +\int_0^T\!\!\!\int_\Omega k\,\nabla u\cdot\nabla v \, \dd \xi \dd t
        +\int_0^T\!\!\!\int_\Omega (c\cdot\nabla u)\, v \, \dd \xi \dd t
        = \int_0^T \iprod{\mu}{v(t)}  \dd t
    \]
    holds for any $ v \in W$.
\end{definition}

Test functions $v \in W$ carry no separate boundary-trace component: since the Dirichlet condition
on $\Sigma_1$ is treated as \emph{essential} data (built into the trial space $U$ via a fixed
extension $\tilde g$), test functions are required to vanish on
$\Sigma_1$ and are otherwise unconstrained on $\Sigma_2$, so that the no-flux condition there
emerges as a \emph{natural} condition. The terminal condition $v(\cdot,T)=0$ is what allows
the initial condition $u(\cdot,0)=u_0$ to be encoded weakly via integration by parts in time
(\cref{def:weak}), without requiring $\partial_t u$ to exist as an independent object
dual to $u$.

\begin{assumption}\label{ass:main}
    We assume throughout that:
    \begin{enumerate}[label=(\roman*)]
        \item $\Omega \subset \R^d$, $1 \le d \le 3$ is a bounded Lipschitz domain, $I=(0,T)$ with $T>0$, and $\Omega_T = \Omega \times I$. The assumption $d\leq 3$ is needed for \cref{lem:Su-lipschitz}, specifically for the elliptic embedding
              $H^2(\Omega)\hookrightarrow C(\overline{\Omega})$,
              which ensures that the adjoint-state pairing with the measure difference $\mu_2-\mu_1$ is well-defined.

        \item The diffusion coefficient $k : \Omega_T \to \R$ satisfies
              \[
                  0 < k_{\min} \le k(\xi, t) \le k_{\max} < \infty
                  \quad\text{for a.e. } (\xi, t)\in\Omega_T,
              \]
              and $k \in L^\infty(\Omega_T)$.

        \item The convection field $c : \Omega_T \to \R^d$ satisfies
              \[
                  c \in L^\infty(\Omega_T; \R^d), \qquad \norm{c}_{L^\infty(\Omega_T; \R^d)} \le R,
              \]
              and, when needed for uniqueness and a priori estimates, we assume
              $\operatorname{div} c \in L^\infty(\Omega_T)$.

        \item $\Gamma_1$ and $\Gamma_2$ are disjoint with $\Gamma_1 \union \Gamma_2 =\partial \Omega$.
              The boundary data and the initial condition satisfy
              \[
                  g \in L^2(\Sigma_1), \qquad u_0 \in L^2(\Omega).
              \]
    \end{enumerate}
\end{assumption}

\begin{theorem}[Existence and uniqueness]
    \label{thm:existence}
    Let \cref{ass:main} hold. In the non-autonomous case, assume additionally that,
for some \(\rho\in(0,1]\),
$
    k\in C^\rho\bigl(\overline I;L^\infty(\Omega)\bigr)
$
and
$
    c\in C^\rho\bigl(\overline I;L^\infty(\Omega)^d\bigr).
$
    For every $(\mu, u_0,g) \in \Masses \times L^2(\Omega) \times L^2(\Sigma_1)$ with $1 < p < d/(d-1)$, the problem
    \eqref{eq:parabolic:convection-diffusion} admits a unique weak solution $u \in L^2(I; W^{1,p}(\Omega))$.
    Moreover, there exist constants $C_p > 0$, independent of $\mu$ and $u_0$, such that
    \[
        \norm{u}_{L^2(I; W^{1,p}(\Omega))} \le C_p \bigl( \norm{\mu}_{\Masses} + \norm{g}_{L^2(\Sigma_1)} + \norm{u_0}_{L^2(\Omega)} \bigr).
    \]
\end{theorem}

\begin{proof}[Sketch of proof]
    The proof follows the strategy of \cite[Theorem~2.2]{casas2013parabolic}, adapted to the
    mixed-boundary, convection--diffusion, and (possibly) non-autonomous setting of
    \cref{ass:main}. We first approximate $\mu \in \Masses$ by continuous functions
    $\mu_k $ in $\Masses$, and solve the regularised problems. We then derive uniform bounds on the resulting sequence $(u_k)_k$ in $L^2\bigl(I;W^{1,p}_{\Gamma_1}(\Omega)\bigr)$ using maximal regularity for the corresponding convection--diffusion operators and pass to the weak limit in the variational formulation. Finally, uniqueness is established by a duality argument using test functions solving the adjoint equation.

    Moreover, since
    $k, c$ may depend on $t$ under \cref{ass:main}, sectoriality (notion recalled in \cref{rem:sectorial} after the proof) of the
    frozen-time operator $E(t)u= -\operatorname{div}\bigl(k(\cdot,t)\nabla u\bigr)+
        c(\cdot,t)\cdot\nabla u $ alone does \emph{not} suffice to obtain maximal
    $L^2$-regularity for
    \[
        \partial_t u_k + E(t) u_k = \mu_k.
    \]
    This requires a non-autonomous argument. We treat the two cases of
    \cref{ass:main} separately:
    \begin{itemize}\item \textbf{Autonomous coefficients} ($k = k(\xi)$, $c = c(\xi)$). Here $E$ is a
              sectorial operator with bounded resolvent on $L^2(\Omega)$, and maximal $L^2\bigl(I; W^{1,p}_{\Gamma_1}(\Omega)\bigr)$-regularity follows from standard semigroup theory (see, \cite[Chap.~4]{Pazy1983}).
        \item \textbf{Non-autonomous coefficients.} When, in addition, $k$ and $c$ satisfy the
              H\"older continuity assumption
              \[
                  k \in C^\rho\bigl(\overline I; L^\infty(\Omega)\bigr), \qquad
                  c \in C^\rho\bigl(\overline I; L^\infty(\Omega)^d\bigr), \qquad \rho \in (0,1],
              \]
              the family $\{E(t)\}_{t \in I}$ generates an evolution family, and maximal
              regularity follows from the theory of Acquistapace and Terreni \cite{acquistapace1987evolution}.
    \end{itemize}
    In both cases, we obtain uniform $L^2(I;W^{1,p}_{\Gamma_1}(\Omega))$ bounds via maximal $L^p$-regularity of the convection-diffusion operator. Passing to the weak limit in the variational formulation yields existence, and uniqueness follows from the energy method using test functions solving the dual equation.
\end{proof}

\begin{remark}[Sectorial operators for autonomous coefficients]
    \label{rem:sectorial}
    The operator
    \[
        E v = -\operatorname{div}(k\nabla v) + c\cdot\nabla v
    \]
    is \emph{sectorial} on $W_0^{1,p}(\Omega)$; that is, its spectrum is contained in the sector
    \[
        \Sigma_{\pi/4}
        =
        \{z\in \mathbb{C} : |\arg(z)| \le \pi/4\}\cup\{0\},
    \]
    and its resolvent is bounded outside larger sectors, i.e., \( \| \lambda - E \|^{-1} \le C/ \| \lambda \| \) for all \(\lambda \) outside any larger sector \(\Sigma_{\phi}\) with \(\phi > \pi/4 \). This follows from the uniform ellipticity condition
    \[
        k \ge k_{\min} >0.
    \]
    As a consequence, the operator $\partial_t + E$ satisfies maximal $L^p$-regularity, that is
    \[
        \partial_t v + E\, v = f \in L^p(\Omega_T)
        \quad \Longrightarrow \quad
        v,\ \partial_t v,\  E \, v \in L^p(\Omega_T),
    \]
    which is essential for establishing Lipschitz continuity of $S_u(\cdot)$. For more on sectorial operators, see \cite[Chap.~2]{Pazy1983}.
\end{remark}

\subsection{Lipschitz continuity of the solution mapping}
\label{sec:lipschitz}
Define the full parameter space as the product space
\[
    \X \defeq \Masses \times L^\infty(\Omega_T) \times L^\infty(\Omega_T; \spaceofdomain),
\]
equipped with the norm
\[
    \norm{\triple}_{\X} \defeq \norm{\mu}_{\Masses} + \norm{k}_{L^\infty(\Omega_T)} + \norm{c}_{L^\infty(\Omega_T; \spaceofdomain)}
    \quad\text{for}\quad
    \triple = (\mu, k, c) \in \X.
\]

\begin{subequations}
    \label{eq:weak-formulation-and-solution-operator}
    Under \cref{ass:main}, the operator $e : U \times \X \to W^*$ is defined by
    \begin{equation}
        \label{eq:weak-formulation}
        \begin{split}
            \dualprod{e(u,\mu,k,c)}{v}_{W^*,W}
             & \defeq \int_{\Omega_T} \bigl( - u\, \partial_t v + k\nabla u\cdot\nabla v
            + (c\cdot\nabla u)\, v \bigr)\,\dx \,\dd t
            \\
             & \qquad - \int_0^T \iprod{\mu}{v(t)}\,\dd t - \int_\Omega u_0 \, v(\cdot, 0) \, \dd \xi
            \\
        \end{split}
    \end{equation}
    for all $v \in W$.
    Furthermore, we introduce the associated solution mapping
    \begin{equation}
        S_u : \X \to U,
        \quad
        \triple = (\mu, k, c) \mapsto u.
    \end{equation}
    Then $u=S_u(\triple)$ solve the weak formulation of \cref{def:weak} if and only if
    \begin{equation}
        \label{eq:solution-operator}
        e(S_u(\triple), \triple) = 0
        \quad \text{in } W^*.
    \end{equation}
\end{subequations}
Next we study the Lipschitz continuity of $S_u$ with respect to the parameters.

\begin{lemma}[Lipschitz dependence of the forward solution operator]
    \label{lem:Su-lipschitz}
    Let \cref{ass:main} and the additional assumption
    $c \in L^\infty(I;W^{1,\infty}(\Omega;\R^d))$ be fulfilled.
    For every $R>0$, let
    \[
        \X_R \defeq \bigl\{ \triple = (\mu,k,c) \in \X
        \,\bigm|\,
        \norm{\mu}_{\Masses} \le R,\ \ k_{\min} \le k \le k_{\max},\ \
        \norm{c}_{L^\infty(I;W^{1,\infty}(\Omega;\R^d))} \le R \bigr\}.
    \]
    Then, for some $L_S = L_S(R) \ge 0$,
    \[
        \norm{S_u(\triple_2)-S_u(\triple_1)}_U \le L_S \norm{\triple_2-\triple_1}_{X}
        \quad\text{for all}\quad
        \triple_1,\triple_2\in \X_R.
    \]
\end{lemma}

We note that this is a statement of \emph{forward} stability: it bounds the change in the state
$u = S_u(\triple)$ produced by a change in the parameters $\triple$. It does not, by
itself, give a bound on $\norm{\triple_2-\triple_1}_\X$ in terms of a difference in
observed data, which would be the corresponding \emph{inverse} stability statement for
the associated inverse problem.

\begin{proof}Let
    $
        \triple_i=(\mu_i,k_i,c_i)\in X_R
    $
    and abbreviate
    $
        u_i=S_u(\triple_i)
    $
    for $i=1,2$.
    Then set $w \defeq u_2 - u_1$.
    Since both states have the same initial datum, $w(\cdot,0)=0$.
    By definition, $u_i$ solves the weak formulation \eqref{eq:weak-formulation}. Subtracting the two equations we obtain, for all $v \in W$ (the boundary term cancels out in the differences),
    \begin{equation}
        \label{eq:diff-identity}
        \begin{split}
            \int_{\Omega_T} \bigl(
            -w\, \partial_t v
            + k_1 \grad w \cdot \grad v + (c_1 \cdot \grad w) \, v \bigr) \dx  \dd t
             & = \int_0^T \iprod{\mu_2 - \mu_1}{v(t)}  \dd t
            \\
            \MoveEqLeft[3]
            + \int_{\Omega_T} \bigl( (k_2 - k_1) \grad u_2 \cdot \grad v
            + ((c_2 - c_1) \cdot \grad u_2) \, v \bigr) \dx  \dd t.
        \end{split}
    \end{equation}
    The derivative now falls on the test function rather than on $w$,
    consistently with the transposition formulation \eqref{eq:weak-formulation}, since
    $\partial_t w$ is not defined on $U$.

    Now, fix $\triple_1$, and consider the dual problem
    \begin{equation}
        \label{eq:pde-dual}
        \begin{split}
             & -\partial_t z - \grad \cdot (k_1 \grad z)
            - c_1 \cdot \grad z - (\grad \cdot c_1)\, z
            = w
            \quad\text{in } \Omega_T,
            \\
             &
               z=0 \text{ on } \Sigma_1,\quad
               k_1\partial_n z-(c_1\cdot n)z=0 \text{ on } \Sigma_2
               ,
            \qquad z(T)=0.
        \end{split}
    \end{equation}

    Under the assumptions that $\Omega$ is a bounded Lipschitz domain, $k_1 \ge k_{\min} > 0$ uniformly elliptic, $c_1 \in L^\infty(\Omega; \R^d)$, and $w \in L^2(\Omega_T)$, \eqref{eq:pde-dual} admits a unique weak solution
    \[
        z \in L^2\bigl(I;W^{1,p'}_{\Gamma_1}(\Omega)\bigr) \cap C(\overline{\Omega_T}).
    \]

    Maximal parabolic regularity for divergence-form operators (see \cite[Section~7.1.3]{evans2010pde})
    implies that the solution operator $w \mapsto z$ is bounded from $L^2(\Omega_T)$ to
    $L^2(I; W^{1,p'}(\Omega))$. Since $p' > d$ by assumption, the Sobolev embedding
    $W^{1,p'}(\Omega) \hookrightarrow C(\overline{\Omega})$ is continuous
    (see \cite[Section~5.6, Theorem~6]{evans2010pde}). Therefore,
    \[
        \norm{z}_{C(\overline{\Omega_T})}
        \le \norm{z}_{L^2(I; C(\overline{\Omega}))}
        \le C \norm{z}_{L^2(I; W^{1,p'}(\Omega))}
        \le C_p \norm{w}_{L^2(\Omega_T)},
    \]
    where $C_p > 0$ depends only on $\Omega$, $T$, $k_{\min}$, and $\norm{c_1}_{L^\infty}$,
    but is independent of the particular choice of $\triple_1 \in \X_R$.

    Choosing $v = z$ in \eqref{eq:diff-identity} and using the dual equation \eqref{eq:pde-dual}, the left-hand side reduces by construction of the adjoint
    to $\int_{\Omega_T} w^2\,\dx\,\dd t$
    \[
        \begin{split}
            \norm{w}_{L^2(\Omega_T)}^2
             &
            \\
             & = \int_0^T \iprod{\mu_2-\mu_1}{z(t)}  \dd t
            + \int_{\Omega_T} \bigl( (k_2 - k_1) \grad u_2 \cdot \grad z
            + ((c_2 - c_1) \cdot \grad u_2) \, z \bigr) \dx  \dd t.
        \end{split}
    \]

    We now estimate the right-hand side term by term.\\
    For the measure difference term, duality between $\Masses$ and
    $C_0(\Omega)$ yields
    \begin{align*}
        \adaptabs{ \int_0^T \iprod{\mu_2-\mu_1}{z(t)}  \dd t }
         & \le \norm{\mu_{2}-\mu_1}_{\Masses} \, \norm{z}_{L^2(I;C_0(\Omega))}
        \\
         & \le C_p \norm{\mu_{2}-\mu_1}_{\Masses} \, \norm{w}_{L^2(\Omega_T)}.
    \end{align*}
    The diffusion coefficient difference term can be estimated by
    \begin{align*}
        \adaptabs{ \int_{\Omega_T} (k_2 - k_1) \grad u_2 \cdot \grad z  \dx  \dd t }
         & \le \norm{k_2 - k_1}_{L^\infty(\Omega_T)}
        \norm{\grad u_2}_{L^2(I;L^p(\Omega))}
        \norm{\grad z}_{L^2(I;L^{p'}(\Omega))}
        \\
         & \le C_p \norm{k_2 - k_1}_{L^\infty(\Omega_T)} \norm{u_2}_{U} \norm{w}_{L^2(\Omega_T)},
    \end{align*}
    using H\"older's inequality.

    Finally, the convection coefficient difference term yields
    \begin{align*}
        \adaptabs{ \int_{\Omega_T} ((c_2 - c_1) \cdot \grad u_2) \, z  \dd \xi \dd t }
         & \le \norm{c_2 - c_1}_{L^\infty(\Omega_T; \spaceofdomain)}
        \norm{\grad u_2}_{L^2(I;L^p(\Omega))}
        \norm{z}_{C(\overline{\Omega_T})}\,|\Omega|^{1/p'}
        \\
         & \le C_p \norm{c_2 - c_1}_{L^\infty(\Omega_T; \spaceofdomain)} \norm{u_2}_{U} \norm{w}_{L^2(\Omega_T)}.
    \end{align*}
    Collecting these estimates, we obtain
    \[
        \norm{w}_{L^2(\Omega_T)}^2
        \le C_p \Bigl( \norm{\mu_{2}-\mu_1}_{\Masses}
        + \norm{k_2 - k_1}_{L^\infty(\Omega_T)} \norm{u_2}_{U}
        + \norm{c_2 - c_1}_{L^\infty(\Omega_T; \spaceofdomain)} \norm{u_2}_{U} \Bigr) \norm{w}_{L^2(\Omega_T)}.
    \]
    If $w \neq 0$, divide by $\norm{w}_{L^2(\Omega_T)}$ to get
    \[
        \norm{w}_{L^2(\Omega_T)}
        \le C_p \Bigl( \norm{\mu_1-\mu_1}_{\Masses}
        + (\norm{k_2 - k_1}_{L^\infty(\Omega_T)} + \norm{c_2 - c_1}_{L^\infty(\Omega_T; \spaceofdomain)}) \norm{u_2}_{U} \Bigr).
    \]
    The estimate above controls only $\norm{w}_{L^2(\Omega_T)}$. To control the full $U$-norm, we return to \eqref{eq:diff-identity} and view it as defining $w$ as a linear functional on $V$. Restricting test functions $v$ in \eqref{eq:diff-identity} to $L^2(I;W^{1,p'}_{\Gamma_1}(\Omega)) $,
    the same three estimates above show that $w$ defines a bounded linear functional on
    $L^2(I;W^{1,p'}_{\Gamma_1}(\Omega))$ with the same bound; since
    $\bigl(L^2(I;W^{1,p'}_{\Gamma_1}(\Omega))\bigr)^* = L^2(I;W^{-1,p}_{\Gamma_1}(\Omega))$, this
    upgrades the estimate to
    \[
        \norm{w}_{U} \le C_p \Bigl( \norm{\mu_2-\mu_1}_{\Masses}
        + (\norm{k_2-k_1}_{L^\infty(\Omega_T)} + \norm{c_2-c_1}_{L^\infty(\Omega_T;\spaceofdomain)})\norm{u_2}_U \Bigr).
    \]

    By the uniform a priori estimate from \cref{thm:existence},
    \[
        \norm{u_2}_{U} \le C_p \bigl( \norm{\mu_2}_{\Masses} + \norm{g}_{L^2(\Sigma_1)} + \norm{u_0}_{L^2(\Omega)} \bigr),
    \]

    and since $\triple_2\in\X_R$, $\norm{\mu_2}_{\Masses}\le R$, so
    $\norm{u_2}_U \le C(R) \defeq C_p\bigl(R+\norm{g}_{L^2(\Sigma_1)}+\norm{u_0}_{L^2(\Omega)}\bigr)$,
    finite and uniform over $\X_R$.

    Combining the upgraded estimate with $\norm{u_2}_U\le C(R)$ and equivalence of $\X$'s
    product norm with its components, we conclude
    \[
        \norm{w}_U \le L_S(R)\norm{\triple_2-\triple_1}_\X
        \quad\text{for}\quad L_S(R) = C_p\bigl(1+C(R)\bigr),
        \quad\text{for all}\quad \triple_1,\triple_2\in\X_R.
        \qedhere
    \]
\end{proof}

\begin{lemma}[Fr\'echet differentiability of the weak form]
    \label{lem:Fret-diff}
    Let \cref{ass:main} hold. The operator $e$ defined in \eqref{eq:weak-formulation} is (partially) Fr\'echet
    differentiable at every point $(u,\mu,k,c) \in U \times \X$ with respect to each argument.
    The partial derivatives are bounded linear operators
    \[
        \diffwrt{e}{u} \in \linear(U; W^*), \quad
        \diffwrt{e}{\mu} \in \linear(\Masses; W^*), \quad
        \diffwrt{e}{k} \in \linear(L^\infty(\Omega_T); W^*), \quad
        \diffwrt{e}{c} \in \linear(L^\infty(\Omega_T;\mathbb{R}^d); W^*),
    \]
    defined for $v  \in W$ by
    \begin{subequations}
        \begin{align}
            \label{eq:diff-u}
            \dualprod{\diffwrt{e}{u} (u,\mu,k,c) \tilde{u}}{v}_{W^*,W}
             & = \int_{\Omega_T} \bigl( -\tilde u\,\partial_t v + k \grad \tilde{u} \cdot \grad v + (c \cdot \grad \tilde{u}) \, v \bigr)  \dd \xi  \dd t,
            \\
            \label{eq:diff-mu}
            \dualprod{\diffwrt{e}{\mu} (u,\mu,k,c) \eta}{v}_{W^*,W}
             & = -\int_0^T \iprod{\eta}{v(t)} \dd t,
            \\
            \label{eq:diff-k}
            \dualprod{\diffwrt{e}{k} (u,\mu,k,c) h}{v}_{W^*,W}
             & = \int_{\Omega_T} h\, \grad u \cdot \grad v \dx\dd t,
            \\
            \label{eq:diff-c}
            \dualprod{\diffwrt{e}{c} (u,\mu,k,c) d}{v}_{W^*,W}
             & = \int_{\Omega_T} (d \cdot \grad u) \, v \dx\dd t,
        \end{align}
    \end{subequations}
    for all $\tilde{u} \in U$, $\eta \in \Masses$, $h \in L^\infty(\Omega_T)$, $d \in L^\infty(\Omega_T; \spaceofdomain)$.
\end{lemma}

\begin{proof}We verify Fréchet differentiability with respect to each argument separately, using the
    definition of \(\dualprod{e(u,\mu,k,c)}{v}\) from \eqref{eq:weak-formulation}, and
    write $v := v_\Omega$ throughout since $e$ has no term dependent on $v_{\Gamma_1}$.

    \textbf{Proof of \eqref{eq:diff-u}.}
    Let \(u_\varepsilon = u + \varepsilon \tilde{u} \) with \(\tilde{u} \in U\). Then for $v = v_\Omega \in W$,
    \begin{align*}
        \dualprod{e(u_\varepsilon,\mu,k,c)}{v}
         &
        = \int_{\Omega_T} \bigl( -(u+\varepsilon\tilde u)\,\partial_t v
        + k \grad(u + \varepsilon \tilde{u}) \cdot \grad v
        + (c \cdot \grad (u + \varepsilon \tilde{u}))\, v \bigr) \dx \dd t
        \\
        \MoveEqLeft[-1]
        - \int_0^T \iprod{\mu}{v(t)} \dd t
        - \int_\Omega u_0\, v(\cdot,0)\,\dx
        \\
        \MoveEqLeft[2]
        = \dualprod{e(u,\mu,k,c)}{v}
        + \varepsilon \Biggl[ \int_{\Omega_T} \bigl( -\tilde{u}\,\partial_t v
                          + k \grad \tilde{u} \cdot \grad v
                          + (c \cdot \grad \tilde{u}) \, v \bigr) \dx \dd t \Biggr].
    \end{align*}
    The initial-data term $-\int_\Omega u_0\,v(\cdot,0)\,\dx$ is independent of $u$, hence cancels in the difference $e(u_\varepsilon,\mu,k,c)-e(u,\mu,k,c)$.

    Dividing by \(\varepsilon\) and letting \(\varepsilon \to 0\) yields \eqref{eq:diff-u}. For continuity, we use the standard characterisation of the norm in \(W^* \):
    \[
        \norm{\diffwrt{e}{u}(u,\mu,k,c) \tilde{u}}_{W^*}
        = \sup_{v \in W, \norm{v}_{W} = 1}
        \dualprod{\diffwrt{e}{u}(u,\mu,k,c) \tilde{u}}{v}.
    \]
    For any \(\tilde{u} \in U\) and \(v \in W\), we estimate
    \begin{align*}
        \adaptabs{ \int_{\Omega_T} -\tilde u\,\partial_t v \dx \dd t}
         &
        \le \norm{\tilde u}_{L^2(I;L^p(\Omega))}
        \norm{\partial_t v}_{L^2(I;L^{p'}(\Omega))},
        \\
        \adaptabs{ \int_{\Omega_T} k \grad \tilde{u} \cdot \grad v \dx \dd t }
         &
        \le \norm{k}_{L^\infty(\Omega_T)} \norm{\grad \tilde{u}}_{L^2(I;L^p(\Omega))}
        \norm{\grad v}_{L^2(I;L^{p'}(\Omega))},
        \quad\text{and}
        \\
        \adaptabs{ \int_{\Omega_T} (c \cdot \grad \tilde{u}) \, v \dx \dd t }
         &
        \le \norm{c}_{L^\infty(\Omega_T; \spaceofdomain)} \norm{\grad \tilde{u}}_{L^2(I;L^p(\Omega))}
        \norm{v}_{C(\overline{\Omega_T})}.
    \end{align*}

    Again, by the Poincar\'e inequality \cite[Corollary 9.19]{brezis2011functional} on \(W^{1,p'}_0(\Omega)\), we have
    $
        \norm{v}_{L^{p'}(\Omega_T)}
        \le C_{P} \norm{\grad v}_{L^{p'}(\Omega_T)},
    $

    so all three terms are bounded by
    \[
        C \bigl( 1
        + (\norm{k}_{L^\infty(\Omega_T)} + \norm{c}_{L^\infty(\Omega_T; \spaceofdomain)}) \bigr)
        \norm{\tilde u}_U \norm{v}_{W},
    \]
    for some \(C>0\) depending only on \(\Omega,\Gamma_1,\Gamma_2,p,T\). Taking the supremum over \(\norm{v}_{W} = 1\)
    gives
    \[
        \norm{\diffwrt{e}{u}(u,\mu,k,c)\tilde{u}}_{W^*}
        \le C \bigl(1 + \norm{k}_{L^\infty(\Omega_T)} + \norm{c}_{L^\infty(\Omega_T; \spaceofdomain)}\bigr) \norm{\tilde{u}}_{U},
    \]
    Hence,
    \[
        \norm{\diffwrt{e}{u}(u,\triple)}
        _{\linear(U;W^*)}
        \leq
        C
        \left(
        1+
        \norm{k}_{L^\infty(\Omega_T)}
        +
        \norm{c}_{L^\infty(\Omega_T;\spaceofdomain)}
        \right).
    \]
    By \cref{ass:main}, the coefficient norms on the right-hand side are
    bounded uniformly over the admissible parameter set. Hence, with
    \[
        M_u\defeq
        C
        \left(
        1+
        \sup_{\triple\in\X}
        \norm{k}_{L^\infty(\Omega_T)}
        +
        \sup_{\triple\in\X}
        \norm{c}_{L^\infty(\Omega_T;\spaceofdomain)}
        \right),
    \]
    we obtain
    \begin{equation}
        \label{eq:eu-bound}
        \norm{\diffwrt{e}{u}(S_u(\triple),\triple)}
        _{\linear(U;W^*)}
        \leq M_u.
    \end{equation}
    Therefore, \(\diffwrt{e}{u} \in \linear(U; W^*)\).

    \textbf{Proof of \eqref{eq:diff-mu}.}
    Let \(\mu_\varepsilon\) be a variation of \(\mu\) in the space \(\Masses\), and denote \(\eta \in \Masses\) so that the directional increment is \(\mu_\varepsilon = \mu + \varepsilon \eta\). Since \(\mu\) is time‑independent, we have
    \[
        \iprod{\mu_\varepsilon}{v(t)}
        = \iprod{\mu}{v(t)} + \varepsilon \iprod{\eta}{v(t)}
        \quad\text{for all} \quad t \in I.
    \]
    Then
    \[
        \begin{split}
            \dualprod{e(u,\mu_\varepsilon,k,c)}{v}
             &
            = \int_{\Omega_T} \bigl( -u\,\partial_t v
            + k \grad u \cdot \grad v
            + (c \cdot \grad u) \, v \bigr) \dd \xi \dd t
            - \int_0^T \iprod{\mu_\varepsilon}{v(t)} \dd t
            - \int_\Omega u_0\, v(\cdot,0)\,\dx
            \\
             &
            = \dualprod{e(u,\mu,k,c)}{v}
            - \varepsilon \int_0^T \iprod{\eta}{v(t)} \dd t.
        \end{split}
    \]
    As in the proof of \eqref{eq:diff-u}, the term $-\int_\Omega u_0 v(\cdot,0)\,\dd \xi$ is
    independent of $\mu$ and cancels identically.
    The expression is linear in \(\varepsilon\), so
    \[
        \iprod{\diffwrt{e}{\mu}(u,\mu,k,c)\,  \eta}{v}
        = -\int_0^T \iprod{\eta}{v(t)} \dd t.
    \]
    Boundedness follows from the fact that
    \[
        \abs{\dualprod{\diffwrt{e}{\mu}(u,\mu,k,c) \, \eta }{v}}
        \le \norm{\eta}_{\mathscr{M}(\Omega)} \norm{v}_{L^2(I;\,C(\overline\Omega))},
    \]
    using the continuous (Morrey) embedding \(W \hookrightarrow L^2(I;\, C(\overline\Omega))\), so
    \(\diffwrt{e}{\mu} \in \linear(\Masses; W^*)\).

    \textbf{Proof of \eqref{eq:diff-k}.}
    Let \(k_\varepsilon = k + \varepsilon h\) with \(h \in L^\infty(\Omega_T)\). Then
    \[
        \dualprod{e(u,\mu,k_\varepsilon,c)}{v}
        = \dualprod{e(u,\mu,k,c)}{v}
        + \varepsilon \int_{\Omega_T} h \grad u \cdot \grad v \dx \dd t.
    \]
    (The $u_0$ term is again independent of $k$ and cancels identically; no boundary
    term is present.)
    Thus
    \[
        \dualprod{\diffwrt{e}{k}(u,\mu,k,c) \, h }{v}
        = \int_{\Omega_T} h \grad u \cdot \grad v \dx \dd t.
    \]
    Boundedness follows from
    \[
        \bigl| \dualprod{\diffwrt{e}{k}(u,\mu,k,c) \, h }{v} \bigr|
        \le \norm{h}_{L^\infty(\Omega_T)}
        \norm{\grad u}_{L^2(I;L^p(\Omega))} \norm{\grad v}_{L^2(I;L^{p'}(\Omega))},
    \]
    so \(\diffwrt{e}{k} \in \linear(L^\infty(\Omega_T); W^*)\).

    \textbf{Proof of \eqref{eq:diff-c}.}
    Let \(c_\varepsilon = c + \varepsilon d\) with \(d \in L^\infty(\Omega_T; \spaceofdomain)\). Then
    \begin{align*}
        \dualprod{e(u,\mu,k,c_\varepsilon)}{v}
         & = \dualprod{e(u,\mu,k,c)}{v}
        + \varepsilon \int_{\Omega_T} (d \cdot \grad u) \,v \dd x \dd t.
    \end{align*}
    (Again, the $u_0$ term is independent of $c$ and cancels identically.)
    Thus
    \[
        \dualprod{\diffwrt{e}{c}(u,\mu,k,c)\, d}{v}
        = \int_{\Omega_T} (d \cdot \grad u) \, v \dd x \dd t,
    \]
    and boundedness follows from
    \[
        \bigl| \dualprod{\diffwrt{e}{c}(u,\mu,k,c) \, d}{v} \bigr|
        \le \norm{d}_{L^\infty(\Omega_T; \spaceofdomain)}
        \norm{\grad u}_{L^2(I;L^p(\Omega))} \norm{v}_{C(\overline{\Omega_T})},
    \]
    so \(\diffwrt{e}{c} \in \linear(L^\infty(\Omega_T; \spaceofdomain); W^*)\).
\end{proof}

The following lemma provides Lipschitz continuity estimates for the partial derivatives of the PDE operator $e$ in weak form.

\begin{lemma}\label{lem:eu-lipschitz}
    Let \cref{ass:main} be satisfied and \(e\) be defined by \eqref{eq:weak-formulation}.
    Then, the partial derivative \(\diffwrt{e}{u}(u, \triple) : U \to W^*\) is uniformly Lipschitz with respect to both \(u \in U\) and \(\triple \in \X\), i.e., there exist constants \(\eulipu ,\eulipx  \ge 0\) such that
    \begin{equation}
        \label{eq:lip-diffe}
        \norm{\diffwrt{e}{u}(u_2,\triple_2) - \diffwrt{e}{u}(u_1,\triple_1)}_{\linear(U; W^*)}
        \le \eulipu \,\norm{u_2 - u_1}_{U}
        + \eulipx \,\norm{\triple_2 - \triple_1}_{\X}
    \end{equation}
    for all \(u_1,u_2 \in U\) and \(\triple_1,\triple_2 \in \X\).
\end{lemma}

\begin{proof}Recall $\diffwrt{e}{u}(u,\triple): U \to W^*$ is defined by
    \[
        \dualprod{\diffwrt{e}{u}(u,\triple) \tilde{u}}{v}_{W^*,W}
        = \int_{\Omega_T} \bigl( -\partial_t \tilde{u}\, v_\Omega
        + k \grad \tilde{u} \cdot \grad v_\Omega + (c \cdot \grad \tilde{u}) \, v_\Omega \bigr) \dx\dd t
        + \int_{\Sigma_1} \tilde{u}|_{\Sigma_1}\, v_{\Gamma_1} \dd \sigma \dd t,
    \]
    for $\tilde{u} \in U$, $v = (v_\Omega, v_{\Gamma_1}) \in W$, where \(\triple = (\mu,k,c)\) and \(\X\) is the parameter space for \(\triple\).
    The operator norm in \(\linear(U; W^*)\) is
    \[
        \norm{\diffwrt{e}{u}(u, \triple)}_{\linear(U; W^*)}
        = \sup_{\tilde{u} \in U, \norm{\tilde{u}}_U= 1}
        \norm{\diffwrt{e}{u}(u, \triple) \, \tilde{u}}_{W^*}.
    \]
    Given pairs \((u_1,\triple_1)\) and \((u_2,\triple_2)\) in \(U \times \X\) and letting $v\defeq v_\Omega$ for brevity, the difference of derivatives satisfies
    \begin{align*}
        \iprod{\bigl(\diffwrt{e}{u}(u_2,\triple_2)
            & - \diffwrt{e}{u}(u_1,\triple_1)\bigr)\,\tilde{u}}{v}
        \\
         & = \int_{\Omega_T} \bigl(
        (k_2 - k_1) \grad \tilde{u} \cdot \grad v
        + (c_2 - c_1) \cdot \grad \tilde{u}\, v
        \bigr) \dx \dd t,
    \end{align*}
    since the term \(-\partial_t w\) is independent of \(u\) and $\triple$. Estimating each term:
    \begin{align*}
        \adaptabs{\int_{\Omega_T} (k_2 - k_1) \grad \tilde{u} \cdot \grad v \dx \dd t }
         & \le \norm{k_2 - k_1}_{L^\infty(\Omega_T)}
        \norm{\grad \tilde{u}}_{L^2(\Omega_T)}
        \norm{\grad v}_{L^2(\Omega_T)},
        \\
        \adaptabs{\int_{\Omega_T} (c_2 - c_1) \cdot \grad  \tilde{u}\, v \dx \dd t }
         & \le \norm{c_2 - c_1}_{L^\infty(\Omega_T; \spaceofdomain)}
        \norm{\grad \tilde{u}}_{L^2(\Omega_T)}
        \norm{v}_{L^2(\Omega_T)}.
    \end{align*}
    By the standard embedding \(W \hookrightarrow L^2(\Omega_T)\) and the Poincaré inequality (for \(W^{1,2}_0(\Omega)\) in the spatial case), we have
    \[
        \norm{\grad \tilde{u}}_{L^2(\Omega_T)}
        \le C\,\norm{\tilde{u}}_{U}
        \quad \text{and} \quad
        \norm{v}_{L^2(\Omega_T)}
        \le C\,\norm{v}_{U}.
    \]
    Thus, for \(k_1,k_2 \in L^\infty(\Omega_T)\) and \(c_1,c_2 \in L^\infty(\Omega_T; \spaceofdomain)\), there exists a constant \(C_E>0\) such that
    \[
        \bigl|\iprod{\bigl(\diffwrt{e}{u}(u_2,\triple_2) - \diffwrt{e}{u}(u_1,\triple_1)\bigr)\,\tilde{u}}{v}\bigr|
        \le
        C_E \bigl(\norm{k_2 - k_1}_{L^\infty(\Omega_T)} + \norm{c_2 - c_1}_{L^\infty(\Omega_T; \spaceofdomain)}\bigr)
        \norm{\tilde{u}}_U\norm{v}_U.
    \]
    Taking the supremum over \(\norm{\tilde{u}}_U= \norm{v}_U = 1\),
    \[
        \norm{\diffwrt{e}{u}(u_2,\triple_2) - \diffwrt{e}{u}(u_1,\triple_1)}_{\linear(U; W^*)}
        \le C_E \bigl(\norm{k_2 - k_1}_{L^\infty(\Omega_T)} + \norm{c_2 - c_1}_{L^\infty(\Omega_T; \spaceofdomain)}\bigr).
    \]
    Since the right‑hand side depends only on the difference in \(\triple =  (\mu, k,c)\), and \(\norm{\triple_2 - \triple_1}_{\X}\) includes the norms \(\norm{k_2 - k_1}_{L^\infty}\) and \(\norm{c_2 - c_1}_{L^\infty}\), and since \(\diffwrt{e}{u}(u, \triple)\) is actually independent of \(u\), the term involving \(\norm{u_2 - u_1}_W\) is absent. Hence \eqref{eq:lip-diffe} holds with \(\eulipu  = 0\) and \(\eulipx  = C_E\).
\end{proof}

The next lemma establishes the boundedness of the invertible operator $\diffwrt{e}{u}(S_u(\triple),\triple)$, which is used to prove the Lipschitz continuity of $\diffwrt{e}{\triple}(S_u(\triple),\triple)$ in \cref{lem:ep-lipschitz}.

\begin{lemma}[Uniform invertibility of the linearized state operator]
    \label{lem:eu-invertible}
    Let \cref{ass:main} hold and let $R>0$. There exists $\beta_R>0$ such that for every
    $\triple \in \X_R$ (defined in \cref{lem:Su-lipschitz}), $\diffwrt{e}{u}(S_u(\triple),\triple) \in \linear(U;W^*)$ is
    boundedly invertible with
    \begin{equation}
        \label{eq:eu-invertible-bound}
        \bignorm{\diffwrt{e}{u}(S_u(\triple),\triple)^{-1}}_{\linear(W^*;U)} \le \frac{1}{\beta_R}
        \qquad \text{for all } \triple \in \X_R.
    \end{equation}
\end{lemma}

\begin{proof}Fix $\triple\in\X_R$ and $f\in W^*$.
    Recalling \eqref{eq:diff-u}, the equation $\diffwrt{e}{u}(S_u(\triple),\triple)\tilde u = f$ is precisely the weak formulation of the linear
    convection--diffusion problem
    \[
        -\partial_t \tilde u - \grad\cdot(k\grad\tilde u) + c\cdot\grad\tilde u = f
        \ \text{ in }\Omega_T, \qquad \tilde u = 0 \text{ on }\Sigma_1,\qquad
        \tilde u(\cdot,0) = 0,
    \]
    with the same ellipticity constant $k_{\min}$ and convection bound
    $\norm{c}_{L^\infty(\Omega_T;\spaceofdomain)}\le R$ for every $\triple\in\X_R$ (this is
    the linearization of \eqref{eq:parabolic:convection-diffusion} about $S_u(\triple)$, and
    inherits linearity and homogeneous data). Existence, uniqueness, and the a priori bound
    \[
        \norm{\tilde u}_U \le \frac{1}{\beta_R}\, \norm{f}_{W^*}
    \]
    follow by the duality argument used in the proof of
    \cref{thm:existence} applied to $f\in W^*$ in place
    of $\mu$, with the same adjoint problem \eqref{eq:pde-dual} and its regularity estimate.
    Since the ellipticity and boundedness constants entering that argument are uniform over
    $\triple \in \X_R$ by definition of $\X_R$, the resulting constant $\beta_R$ is likewise
    uniform over $\X_R$.
\end{proof}

\subsection{The reduced adjoint equation}
\label{sec:adjoint}

In this subsection, we introduce the adjoint equation and its reduced form, which are essential for differentiating the data term (\cref{sec:parabolic:dataterm}). We also present the corresponding Lipschitz properties of the adjoint operator, which will be used to derive Lipschitz estimates for the reduced gradient.

We recall that
\[
    F(\triple) \defeq J(S_u(\triple), \triple),
\]
where, following \eqref{eq:weak-formulation-and-solution-operator}, $S_u: \X \to U$ arises from the satisfaction of
\begin{equation}
    \label{eq:dataterm:esu}
    e(S_u(x), x)=0.
\end{equation}
We will now construct Lipschitz estimates for $F'$.

By \cref{lem:Fret-diff}, $e(\cdot,\triple)$ is Fréchet differentiable in $u$. Since $\diffwrt{e}{u}(S_u(\triple),\triple): U \to W^*$ is boundedly invertible by \cref{prop:adjoint-inverse}, the implicit function theorem (see \cite[Thm.~4.3]{zeidler1986nonlinear}) shows that $S_u: \X \to U$ is Fréchet differentiable at $\triple$ with
\[
    S_u'(\triple) = - \diffwrt{e}{u}(S_u(\triple),\triple)^{-1} \diffwrt{e}{\triple}(S_u(\triple),\triple),
\]
that is
\begin{equation}
    \label{eq:dataterm:basic-adjoint}
    \diffwrt{e}{u}(S_u(\triple), \triple)S_u'(\triple) + \diffwrt{e}{\triple}(S_u(\triple), \triple) = 0,
\end{equation}
where we can use \cref{lem:Fret-diff} to compute $\diffwrt{e}{u}(u, \triple) \in \linear(U; W^*)$ and
\[
    \diffwrt{e}{\triple}(u, \triple) =
    \bigl(
    \diffwrt{e}{\mu}(u, \triple),
    \diffwrt{e}{k}(u, \triple),
    \diffwrt{e}{c}(u, \triple)
    \bigr)
    \in \linear(\X; W^*).
\]
 The total derivative of \(F\) with respect to the parameter triple \(\triple =  (\mu,k,c)\) is now obtained from the chain rule
\begin{equation}
    \label{eq:dataterm:fprime}
    F'(\triple) = \diffwrt{J}{\triple}(S_u(\triple), \triple) + \diffwrt{J}{u}(S_u(\triple), \triple)S_u'(\triple).
\end{equation}
\Cref{lem:reduced-adjoint} will show that there exists a unique adjoint variable $w_\triple \in W$ such that
\begin{equation}
    \label{eq:dataterm:reduced-adjoint}
    w_\triple \diffwrt{e}{u}(S_u(\triple), \triple) = \diffwrt{J}{u}(S_u(\triple), \triple)
    \quad\text{in } U^*.
\end{equation}
Applying $w_\triple$ to both sides of \eqref{eq:dataterm:basic-adjoint} and using \eqref{eq:dataterm:reduced-adjoint} yields
\begin{equation}
    \label{eq:dataterm:reduced-adjoint-reco}
    \diffwrt{J}{u}(S_u(\triple), \triple)S_u'(\triple) + w_\triple \diffwrt{e}{\triple}(S_u(\triple), \triple) = 0.
\end{equation}
Using this expression in \eqref{eq:dataterm:fprime} establishes
\[
    F'(x) = \diffwrt{J}{\triple}(S_u(\triple), \triple) - w_\triple\, \diffwrt{e}{\triple}(S_u(\triple), \triple).
\]
We call \eqref{eq:dataterm:reduced-adjoint} the \term{reduced adjoint equation}, as, compared to solving the operator $S'_u(x)$ from the \term{basic adjoint equation} \eqref{eq:dataterm:basic-adjoint}, the dimension of the unknown $w_\triple$ is reduced in discretisations.

\begin{remark}[Classical adjoint equation]
    \label{rem:classical-adjoint}
    The abstract adjoint $w_\triple\in W$ from \eqref{eq:dataterm:reduced-adjoint} solves the classical
    backward equation
    \begin{equation}
        \label{eq:adjoint}
        \begin{aligned}
            -\partial_t w_\triple - \grad \cdot (k \grad w_\triple)- \grad \!\cdot (c\,w_\triple) & = \grad_u J(S_u(\triple), \triple) &  & \text{in } \Omega_T,
            \\
            w_\triple                                                                             & = 0                                &  & \text{on } \Sigma_1,
            \\
            k \frac{\partial w_\triple}{\partial n} - (c \cdot n) w_\triple         & = 0                                &  & \text{on } \Sigma_2,
            \\
            w_\triple(\cdot,T)                                                                    & = 0                                &  & \text{in } \Omega .
        \end{aligned}
    \end{equation}
    This follows by testing $\diffwrt{e}{u}(u,\mu,k,c)\tilde u$ with $w_\triple$ and
    computing $\dualprod{\diffwrt{J}{u}(S_u(\triple),\triple)}{\tilde u}$ for every
    $\tilde u \in U$ with $\tilde u = 0$ on $\Sigma_1$ and $\tilde u(\cdot,0)=0$ (the
    admissible perturbation directions, since $g$ and $u_0$ are fixed data). Since $\tilde u$
    vanishes on $\Sigma_1$, integrating by parts produces a boundary integral only on
    $\Gamma_2$, where $\tilde u$ is unconstrained; the vanishing of its coefficient there
    yields the mixed boundary condition on $\Sigma_2$. The condition $w_\triple = 0$ on
    $\Sigma_1$ is imposed independently, matching the essential condition satisfied by test
    functions in $W$ (see \cref{eq:weak-formulation}), consistent with $w_\triple$'s role as a
    transposition test function.
\end{remark}

\begin{remark}
    The reduced gradient \eqref{eq:dataterm:reduced-adjoint} requires computing $  w_\triple \diffwrt{e}{\triple}(S_u(\triple), \triple)$ where $w_\triple \in W$ solves the adjoint equation \eqref{eq:adjoint}. For perturbations $(\eta,h,d) \in \X$, this action takes the explicit vector form:
    \[
        w_\triple \, \diffwrt{e}{\triple}(S_u(\triple), \triple)(\eta,h,d)
        =
        \begin{pmatrix}
            -\iprod{\bar{w}_\triple}{\eta}
            \\
            \int_{\Omega_T} h \, \grad u \cdot \grad w_\triple \dx \dd t
            \\
            \int_{\Omega_T} (d \cdot \grad u)\, w_\triple \dx \dd t
        \end{pmatrix}
        ,
    \]
    where
    \[
        \bar{w}_\triple(\spatial) \defeq \int_I w_\triple(\spatial, t)\dd t,
    \]
    and the measure term is
    \[
        - \int_{\Omega_T} w_\triple \d(\eta \otimes t) = -\iprod{\bar{w}_\triple}{\eta}.
    \]
    \Cref{sec:implementation} shows that these terms arise naturally from the adjoint solve, enabling efficient gradient computation via matrix-free operators. Further details are available in our software implementation \cite{dangvalkonen2026implementation}.
\end{remark}

\begin{lemma}[Lipschitz continuity of $\diffwrt{e}{\triple}(S_u(\triple), \triple)$]
    \label{lem:ep-lipschitz}
    Suppose \cref{ass:main} holds, let $R>0$, and let $\triple_1,\triple_2 \in \X_R$ (defined in \cref{lem:Su-lipschitz}).
    Let
    \begin{equation*}
        \diffwrt{e}{\triple}(S_u(\triple), \triple): \X \to W^*, \quad
        \diffwrt{e}{\triple}(S_u(\triple), \triple) h := \left[ \triple \mapsto \diffwrt{e}{\triple}(S_u(\triple), \triple) \right]'(\triple) h,
    \end{equation*}
    denote the derivative (at $\triple$) of the map $\triple \mapsto \diffwrt{e}{\triple}(S_u(\triple), \triple)$.
    Then, on $\X_R$, it is Lipschitz continuous, i.e.,
    \[
        \left\Vert \diffwrt{e}{\triple}(S_u(\triple_2),\triple_2) - \diffwrt{e}{\triple}(S_u(\triple_1),\triple_1) \right\Vert_{\linear(\X; W^*)}
        \le L_e(R) \norm{\triple_2-\triple_1}_{\X}.
    \]
    for all $\triple_1,\triple_2\in\X_R$, where $L_e(R)\ge0$ depends only on bounds from \cref{ass:main}  and on $R$.
    Consequently,
    \[
        M_e(R) \defeq \sup_{\triple \in \X_R} \left\Vert \diffwrt{e}{\triple}(S_u(\triple), \triple) \right\Vert_{\linear(\X; W^*)} < \infty,
    \]
    and $\triple \mapsto \diffwrt{e}{\triple}(S_u(\triple), \triple)$ is Lipschitz continuous from $\X_R$ to $\linear(\X; W^*)$.
\end{lemma}

\begin{proof}We have above proved the Fréchet differentiability of $S_u$.
    By the chain rule for Fréchet derivatives, we have
    \[
        [\triple \mapsto e(S_u(\triple), \triple)]'(\triple) = \diffwrt{e}{\triple}(S_u(\triple), \triple) + \diffwrt{e}{u}(S_u(\triple), \triple)\circ S_u'(\triple).
    \]
    Since $e(S_u(\triple), \triple)\equiv 0$, the total derivative vanishes, i.e.,
    \[
        0 = \diffwrt{e}{\triple}(S_u(\triple), \triple) + \diffwrt{e}{u}(S_u(\triple), \triple) S_u'(\triple).
    \]
    Rearranging and using \cref{lem:eu-invertible} to invert $\diffwrt{e}{u}$ gives
    \[
        \diffwrt{e}{\triple}(S_u(\triple), \triple) = -\diffwrt{e}{u}(S_u(\triple), \triple) S_u'(\triple), \qquad
        S_u'(\triple) = -\diffwrt{e}{u}(S_u(\triple),\triple)^{-1}\diffwrt{e}{\triple}(S_u(\triple),\triple).
    \]
    Let $\triple_1,\triple_2\in\X_R$. Then
    \begin{align*}
         & \diffwrt{e}{\triple}(S_u(\triple_2),\triple_2) - \diffwrt{e}{\triple}(S_u(\triple_1),\triple_1)
        \\
         & = -\diffwrt{e}{u}(S_u(\triple_2),\triple_2)S_u'(\triple_2) + \diffwrt{e}{u}(S_u(\triple_1),\triple_1)S_u'(\triple_1)
        \\
         & = -\bigl[\diffwrt{e}{u}(S_u(\triple_2),\triple_2)S_u'(\triple_2) - \diffwrt{e}{u}(S_u(\triple_2),\triple_2)S_u'(\triple_1)\bigr]
        \\
        \MoveEqLeft[-1]
        - \bigl[\diffwrt{e}{u}(S_u(\triple_2),\triple_2) - \diffwrt{e}{u}(S_u(\triple_1),\triple_1)\bigr]S_u'(\triple_1).
    \end{align*}
    We have
    \[
        \begin{split}
            \bignorm{\diffwrt{e}{\triple}(S_u(\triple_2),\triple_2) - \diffwrt{e}{\triple}(S_u(\triple_1),\triple_1)}_{\linear(\X; W^*)}
             &
            \le
            \norm{\diffwrt{e}{u}(S_u(\triple_2),\triple_2)}_{\linear(U; W^*)}\norm{S_u'(\triple_2)-S_u'(\triple_1)}_{\linear(\X; U)}
            \\
            \MoveEqLeft[2]
            + \norm{S_u'(\triple_1)}_{\linear(\X; U)}\bignorm{\diffwrt{e}{u}(S_u(\triple_2),\triple_2)-\diffwrt{e}{u}(S_u(\triple_1),\triple_1)}_{\linear(U; W^*)}.
        \end{split}
    \]
     Writing $e_u(\triple) \defeq
        \diffwrt{e}{u}(S_u(\triple),\triple)$, $e_\triple(\triple)\defeq
        \diffwrt{e}{\triple}(S_u(\triple),\triple)$ for brevity, we have
    \[
        S_u'(\triple_2) - S_u'(\triple_1)
        = -e_u(\triple_2)^{-1}\bigl[e_\triple(\triple_2)-e_\triple(\triple_1)\bigr]
        - \bigl[e_u(\triple_2)^{-1}-e_u(\triple_1)^{-1}\bigr] e_\triple(\triple_1).
    \]
    By the resolvent identity $A^{-1}-B^{-1} = A^{-1}(B-A)B^{-1}$, we get
    \[
        \begin{split}
            \adaptnorm{e_u(\triple_2)^{-1}-e_u(\triple_1)^{-1}}_{\linear(W^*;U)}
             & \le \adaptnorm{e_u(\triple_2)^{-1}}_{\linear(W^*;U)}
            \adaptnorm{e_u(\triple_2)-e_u(\triple_1)}_{\linear(U;W^*)}
            \adaptnorm{e_u(\triple_1)^{-1}}_{\linear(W^*;U)}
            \\
             & \le \frac{1}{\beta_R^2}\,\eulipx\, \norm{\triple_2-\triple_1}_\X,
        \end{split}
    \]
    using \cref{lem:eu-invertible} for the two outer factors and \cref{lem:eu-lipschitz} for
    the middle one.

    Since $\diffwrt{e}{\mu}$ is independent of $\triple$
    (see \eqref{eq:diff-mu}) and $\diffwrt{e}{k},\diffwrt{e}{c}$ depend on $\triple$ only through
    $u=S_u(\triple)$, linearly (\cref{eq:diff-k,eq:diff-c}), the estimates of
    \cref{lem:Fret-diff} combined with \cref{lem:Su-lipschitz} give directly
    \begin{equation}
        \label{eq:ex-diff-bound}
        \norm{e_\triple(\triple_2)-e_\triple(\triple_1)}_{\linear(\X;W^*)}
        \le L_{e_\triple}\,\norm{\triple_2-\triple_1}_\X, \qquad L_{e_\triple} = C(\Omega,p)\,L_S(R),
    \end{equation}
    which we use below only as an ingredient for bounding $S_u'(\triple_2)-S_u'(\triple_1)$; the
    substantive content of this lemma is the resulting Lipschitz bound on $S_u'$ itself
    \eqref{eq:Su-diff-bound}, needed in \cref{cor:f-lipschitz-differentiable}.

    Together with \cref{eq:eu-invertible-bound}, the Lipschitz bound \cref{eq:ex-diff-bound}, and the bound $\norm{e_\triple(\triple_1)}_{\linear(\X;W^*)}\le
        M_{e}(R)$ on $\X_R$, we obtain
    \begin{equation}
        \label{eq:Su-diff-bound}
        \bignorm{S_u'(\triple_2)-S_u'(\triple_1)}_{\linear(\X;U)}
        \le \frac{1}{\beta_R} L_{e_\triple}\,\norm{\triple_2-\triple_1}_\X
        + \frac{1}{\beta_R^2}\eulipx\, M_e(R)\,\norm{\triple_2-\triple_1}_\X
        =: L_{S'}(R)\,\norm{\triple_2-\triple_1}_\X.
    \end{equation}
    By \cref{lem:Su-lipschitz} (on $\X_R$),
    \[
        \norm{S_u(\triple_2) - S_u(\triple_1)}_U \le L_S(R) \norm{\triple_2 - \triple_1}_{\X},
    \]
    and by \cref{lem:eu-lipschitz},
    \begin{equation}
        \label{eq:eu-diff-bound}
        \norm{ \diffwrt{e}{u}(S_u(\triple_2), \triple_2) - \diffwrt{e}{u}(S_u(\triple_1), \triple_1)}_{\linear(U; W^*)}
        \le \eulipx  \norm{\triple_2 - \triple_1}_{\X}.
    \end{equation}
    By \eqref{eq:eu-bound}, it follows that
    \[
        \bignorm{S_u'(\triple)}_{\linear(\X;U)} \le
        L_S(R)
        \quad\text{and}\quad
        \norm{\diffwrt{e}{u}(S_u(\triple), \triple)}_{\linear(U; W^*)} \le M_u.
    \]
    These estimates imply that
    \[
        L_e(R) \le M_u\, L_{S'}(R) + L_S(R)\, \eulipx,
    \]
    where the first term comes from the  bound on
    $\norm{S_u'(\triple_2)-S_u'(\triple_1)}$ derived in \cref{eq:Su-diff-bound}, and the second from the Lipschitz
    continuity of $\diffwrt{e}{u}$ as in \cref{eq:eu-diff-bound}.
    Boundedness of $\X_R$ (\cref{lem:Su-lipschitz}) ensures that all constants
    $\beta_R,L_S(R),L_{e_\triple},M_{e}(R),M_u$ are finite and uniform on $\X_R$.
\end{proof}

\begin{proposition}[Uniform adjoint invertibility]
    \label{prop:adjoint-inverse}
    Suppose \cref{ass:main} is satisfied and let \( e \) be defined by \eqref{eq:weak-formulation}. For $\triple \in \X$, set
    \[
        e_u(\triple) \defeq \diffwrt{e}{u}(S_u(\triple), \triple) : U \to  W^*.
    \]
    Then, the adjoint operator
    \[
        e_u(\triple)^* : W \to U^{*}
    \]
    is uniformly boundedly invertible. More precisely, there exists a constant \( C_A > 0 \), independent of $\triple \in \X$, such that
    \[
        \norm{ [e_u(\triple)^*]^{-1} }_{\linear(W; U^{*})} \le C_A.
    \]
\end{proposition}

\begin{proof}Let $E(\triple): W^{1,p}_0(\Omega) \to (W^{1,p}_0(\Omega))^*$ be defined weakly by
    \[
        \dualprod{E (\triple)v}{w}
        \defeq \iprod{k\nabla v}{\nabla w} + \dualprod{c\cdot\nabla v}{w},
        \quad\text{for all}\quad
        v,\, w \in W^{1,p}_0(\Omega),
    \]
    where $k \ge k_{\min} > 0$ and $c \in L^\infty(\Omega; \spaceofdomain)$. Then $e_u(\triple)[v]=\int_{\Omega_T}\bigl(-\partial_t v\,\phi+k\grad v\cdot\grad\phi+(c\cdot\grad v)\phi\bigr) \dx \dd t$ as derived in \eqref{eq:diff-u}, is the weak form of $\partial_t v+ E(\triple)\, v=0$.
    The operator $E(\triple)$ is sectorial on $L^p(\Omega)$ under uniform ellipticity
    and suitable assumptions on the coefficients (see Remark~\ref{rem:sectorial}). Hence, $\partial_t + E (\triple)$ admits maximal $L^p$-regularity (see \cite[Thm.~4.3.1]{lunardi2011analytic}). That is,
    \[
        v, \partial_t v, E (\triple) v \in L^p(\Omega_T) \quad \text{for all} \, f \in L^p(\Omega_T).
    \]
    and
    \[
        \norm{v}_{L^p(\Omega_T)}+\norm{\partial_t v}_{L^p(\Omega_T)}+\norm{E(\triple)v}_{L^p(\Omega_T)}\le C\norm{f}_{L^p(\Omega_T)}.
    \]
    The adjoint $e_u(\triple)^*$ solves the backward equation $-\partial_t w_\triple+ E(\triple)^*w_\triple=f$. Since $E(\triple)$ is sectorial on $L^p(\Omega)$, its adjoint $E (\triple)^*$ is also sectorial on $L^{p'}(\Omega)$
    (see \cite[Chap.~1--2]{Pazy1983}), so by maximal regularity (see \cite[Thm.~4.3.1]{lunardi2011analytic}) and \cref{rem:sectorial})
    gives the uniform estimates
    \[
        \norm{w_\triple}_U+\norm{\partial_t w_\triple}_U+\norm{E(\triple)^*w_\triple}_U\le C\norm{f}_U,
    \]
    with $C$ independent of $\triple \in \X$ (depends on $k_{\min}$, $|c|_{L^\infty}$, $\Omega$).
    Thus, for $e_u(\triple)^*w_\triple=f$,
    \[
        \norm{w_\triple}_U\le C\norm{f}_U\implies\bignorm{[e_u(\triple)^*]^{-1}}_{\linear(W; U)}\le C=C_A.
        \qedhere
    \]
\end{proof}

\begin{lemma}[Solvability of the reduced adjoint equation]
    \label{lem:reduced-adjoint}
    Suppose \cref{ass:main} holds and let $e(u(\triple), \triple)$ with $ u \defeq S_u(\triple)$ be defined as in \eqref{eq:weak-formulation}.
    Assume the data fidelity $J: U \times \X \to \mathbb{R}$ is (partially) Fréchet differentiable w.r.t. $u$ with
    \[
        \diffwrt{J}{u}(u, \triple) \in U^* \quad \text{for all } (u, \triple) \in U \times \X.
    \]
    Then for every $\triple \in \X$, there exists a unique adjoint state $w_\triple \in W$ solving the reduced adjoint equation
    \begin{equation}
        \label{eq:reduced-adjoint}
        w_\triple \diffwrt{e}{u}(S_u(\triple), \triple) = \diffwrt{J}{u}(S_u(\triple), \triple)
        \quad\text{in } U^*.
    \end{equation}
\end{lemma}

\begin{proof}By \cref{lem:Fret-diff}, $e(\cdot,\triple)$ is Fréchet differentiable in $u$ at $u \defeq S_u(\triple)$,
    so $\diffwrt{e}{u}(S_u(\triple),\triple): U \to W^*$ is bounded linear. By \cref{prop:adjoint-inverse},
    its adjoint
    \[
        T(\triple) \defeq \diffwrt{e}{u}(S_u(\triple),\triple)^*: W \to U^*
    \]
    is \emph{uniformly boundedly invertible}, hence $T(\triple)^{-1}: U^* \to W$ exists and is bounded.

    Let
    \[
        w_\triple \defeq T(\triple)^{-1} \diffwrt{J}{u}(S_u(\triple),\triple) \in W.
    \]
    By definition of $T(\triple)$, we obtain the dual form
    \[
        T(\triple) w_\triple = \diffwrt{e}{u}(S_u(\triple),\triple)^* w_\triple = \diffwrt{J}{u}(S_u(\triple),\triple) \quad \text{in } U^*.
    \]
    The primal form \eqref{eq:reduced-adjoint} follows immediately by definition of adjoint operator.
    By \cref{prop:adjoint-inverse}, $T(\triple)$ is uniformly boundedly invertible, i.e., there exists $C_A>0$,
    independent of $\triple\in \X$, such that
    \[
        \norm{T(\triple)^{-1} v}_{W} \le C_A \norm{v}_{U^*} \quad\text{for all}\quad v\in U^*.
    \]
    For uniqueness, suppose $T(\triple) w_1 = T(\triple) w_2$. Then $T(\triple)(w_1-w_2)=0$, so
    \[
        w_1-w_2 = T(\triple)^{-1} \bigl( T(\triple)(w_1-w_2) \bigr) = 0.
    \]
    Thus $T(\triple)$ is injective. It follows that $w_\triple$ is unique in $W$.
\end{proof}

\begin{example}[Quadratic data fidelity terms]
    \label{ex:quadratic-J}
    A standard choice of quadratic data fidelity satisfying \cref{lem:reduced-adjoint}:
    \[
        J(u, \triple) = \frac{1}{2} \norm{Au - z_d}_Y^2, \quad A \in \linear(U; Y),\ z_d \in Y,
    \]
    where $Y$ is a Hilbert observation space (e.g., $L^2(\Omega_T)$). The Fréchet derivative with respect to $u$ is given by
    \[
        \diffwrt{J}{u}(u, \triple) = A^*(Au - z_d) \in U^*.
    \]
    Moreover, it satisfies the bound
    \[
        \norm{\diffwrt{J}{u}}_{U^*} \le \norm{A^*}_{\linear(Y;U^*)} \bigl( \norm{A}_{\linear(U;Y)} \norm{u}_U + \norm{z_d}_Y \bigr) \le M_J < \infty.
    \]
    In addition, the derivative is Lipschitz continuous with respect to $u$:
    \[
        \norm{\diffwrt{J}{u}(u_2,\triple_2) - \diffwrt{J}{u}(u_1,\triple_1)}_{U^*} \le \norm{A^*}_{\linear(Y;U^*)}
        \norm{A}_{\linear(U;Y)}\norm{u_2-u_1}_U .
    \]
    This quadratic data fidelity satisfies all assumptions of \cref{lem:reduced-adjoint} (Fréchet differentiability $\diffwrt{J}{u} \in U^*$ and boundedness).
\end{example}

\begin{lemma}[Lipschitz continuity of the adjoint operator]
    \label{lem:A-lipschitz}
    Assume that \cref{ass:main} holds and that \( e \) is defined by \eqref{eq:weak-formulation}.
    Let
    \[
        e_u(\triple) \defeq \diffwrt{e}{u}(S_u(\triple), \triple) : U \to W^*.
    \]
    Then there exists a constant \( L_A \ge 0 \) such that
    \[
        \norm{ e_u(\triple_2)^* - e_u(\triple_1)^* }_{\linear(W; U^{*})}
        \le L_A \norm{\triple_2 - \triple_1}_{\X}
        \quad \text{for all } \triple_1, \triple_2 \in \X.
    \]
\end{lemma}

\begin{proof}By definition,
    \[
        e_u(\triple_2)^* - e_u(\triple_1)^*
        = \diffwrt{e}{u}(S_u(\triple_2), \triple_2)^* - \diffwrt{e}{u}(S_u(\triple_1), \triple_1)^*.
    \]
    Since the adjoint map is isometric, we have
    $
        \norm{T^*}_{\linear(W; U)} = \norm{T}_{\linear(U; W^*)}.
    $
    Hence,
    \[
        \norm{ e_u(\triple_2)^* - e_u(\triple_1)^* }_{\linear(W; U^{*})}
        = \norm{ \diffwrt{e}{u}(S_u(\triple_2), \triple_2)
            - \diffwrt{e}{u}(S_u(\triple_1), \triple_1) }_{\linear(U; W^*)}.
    \]
    By \cref{lem:eu-lipschitz},
    \[
        \norm{ \diffwrt{e}{u}(S_u(\triple_2), \triple_2)- \diffwrt{e}{u}(S_u(\triple_1), \triple_1) }_{\linear(U; W^*)}
        \le \eulipu  \norm{S_u(\triple_2) - S_u(\triple_1)}_U+ \eulipx  \norm{\triple_2 - \triple_1}_{\X}.
    \]
    By \cref{lem:Su-lipschitz},
    \[
        \norm{S_u(\triple_2) - S_u(\triple_1)}_U \le L_S \norm{\triple_2 - \triple_1}_{\X}.
    \]
    Thus,
    \[
        \norm{ e_u(\triple_2)^* - e_u(\triple_1)^* }_{\linear(W; U^{*})}
        \le \eulipu  L_S \norm{\triple_2 - \triple_1}_{\X}+ \eulipx  \norm{\triple_2 - \triple_1}_{\X}
        = (\eulipu  L_S + \eulipx )\norm{\triple_2 - \triple_1}_{\X}.
    \]
    Setting \( L_A \defeq \eulipu  L_S + \eulipx \) completes the proof.
\end{proof}

\subsection{The differential of the data term}
\label{sec:parabolic:dataterm}

Recalling the expression \eqref{eq:dataterm:fprime}, we now proceed with the Lipshitz estimates for $F'$. We require:

\begin{assumption}[Data fidelity regularity]
    \label{ass:dataterm:ju-lip}
    The data fidelity $J: U \times \X \to \R$ is Fréchet differentiable with
    $\diffwrt{J}{u}: U \times \X \to U^*$ and $\diffwrt{J}{\triple}: U \times \X \to \X^*$
    Lipschitz continuous:
    \begin{align*}
        \norm{\diffwrt{J}{u}(u_2, \triple_2) - \diffwrt{J}{u}(u_1, \triple_1)}_{U^*}
         & \le \julipu \norm{u_{2} - u_1}_U + \julipx \norm{\triple_2 - \triple_1}_{\X},
        \quad\text{and}
        \\
        \norm{\diffwrt{J}{\triple}(u_2, \triple_2) - \diffwrt{J}{\triple}(u,\triple)}_{\X^*}
         &
        \le \jxlipu \norm{u_{2} - u_1}_U + \jxlipx \norm{\triple_2 - \triple_1}_\X.
    \end{align*}

    for all $u_1,u_2 \in U$ and $\triple_1,\triple_2 \in \X$
    for some $\julipu,\julipx,\jxlipu,\jxlipx \ge 0$.
\end{assumption}

We first estimate the first factor of the term $\diffwrt{J}{u}(S_u(\triple), \triple)S_u'(\triple)$ in \eqref{eq:dataterm:fprime}.

\begin{lemma}[Lipschitz continuity of the Jacobian]
    \label{lemma:dataterm:ju-lip}
    Suppose \cref{ass:main,ass:dataterm:ju-lip} hold.
    Then, the mapping $\triple \mapsto \diffwrt{J}{u}(S_u(\triple), \triple): \X\to U^*$ is Lipschitz continuous, i.e.,
    \[
        \norm{\diffwrt{J}{u}(S_u(\triple_2),\triple_2) - \diffwrt{J}{u}(S_u(\triple_1),\triple_1)}_{U^*} \le \jfullulipx \norm{\triple_2-\triple_1}_{\X}, \quad \text{for all}\quad \triple_1,\triple_2\in\X,
    \]
    where $\jfullulipx = \julipu L_S + \julipx>0$ with $L_S$ from \cref{lem:Su-lipschitz}.
\end{lemma}

\begin{proof}Immediate consequence of \cref{ass:dataterm:ju-lip,lem:Su-lipschitz}.
\end{proof}

\begin{lemma}[Boundedness of the differential of the data fidelity]
    \label{lem:data-bound}
    Suppose \cref{ass:main,ass:dataterm:ju-lip} hold. Then there exists a constant $M_J < \infty$ such that
    \[
        \norm{\diffwrt{J}{u}(S_u(\triple), \triple)}_{U^*} \le M_J
        \quad\text{for all}\quad \triple \in \X.
    \]
\end{lemma}

\begin{proof}By \cref{ass:dataterm:ju-lip}, $\diffwrt{J}{u}: U \times \X \to U^*$ is Lipschitz continuous:
    \[
        \norm{\diffwrt{J}{u}(u_2, \triple_2) - \diffwrt{J}{u}(u_1, \triple_1)}_{U^*}
        \le \julipu \norm{u_2 - u_1}_U + \julipx \norm{\triple_2 - \triple_1}_{\X}.
    \]
    By maximal $L^p$-regularity (\cite[Thm.~4.3.1]{lunardi2011analytic}) and \cref{ass:main}, the solution operator $S_u: \X \to U$ is uniformly bounded, i.e.,
    \[
        \norm{S_u(\triple)}_U \le C_S \quad \text{for all } \triple \in \X.
    \]
    Fix a reference point $(\hat{u}, \hat{\triple}) \in U \times \X$. Then
    \begin{align*}
        \norm{\diffwrt{J}{u}(S_u(\triple), \triple)}_{U^*}
         & \le \norm{\diffwrt{J}{u}(S_u(\triple), \triple) - \diffwrt{J}{u}(\hat{u}, \hat{\triple})}_{U^*}
        + \norm{\diffwrt{J}{u}(\hat{u}, \hat{\triple})}_{U^*}
        \\
         & \le \julipu \norm{S_u(\triple) - \hat{u}}_U + \julipx \norm{\triple - \hat{\triple}}_{\X}
        + \norm{\diffwrt{J}{u}(\hat{u}, \hat{\triple})}_{U^*}.
    \end{align*}
    Since $S_u$ is uniformly bounded, $S_u(\triple)$ stays in a bounded set in $U$. For bounded $\triple \in \X$, we obtain
    \[
        \norm{\diffwrt{J}{u}(S_u(\triple), \triple)}_{U^*} \le M_J := \julipu (C_S + \norm{\hat{u}}_U) + \julipx \diam(\X) + \norm{\diffwrt{J}{u}(\hat{u}, \hat{\triple})}_{U^*} < \infty.
        \qedhere
    \]
\end{proof}

Finally, we are able to obtain a Lipschitz estimate for the remaining second term of \eqref{eq:dataterm:fprime}.

\begin{theorem}[Partial Lipschitz estimate]
    \label{thm:wp-ep-lipschitz}
    Suppose \cref{ass:main,ass:dataterm:ju-lip} hold.
    Then the mapping $\triple \mapsto \diffwrt{J}{u}(S_u(\triple), \triple)S'(\triple)$ is Lipschitz continuous from $\X$ to $\X^*$:
    there exists $C_{\mathrm{tot}}^{\mathrm{Lip}}>0$ such that
    \begin{equation*}
        \bignorm{
            \diffwrt{J}{u}(S_u(\triple_1), \triple_1)S'(\triple_1)
            - \diffwrt{J}{u}(S_u(\triple_2), \triple_2)S'(\triple_2)
        }_{\X^*}
        \le C_{\mathrm{tot}}^{\mathrm{Lip}}\norm{\triple_2-\triple_1}_{\X}
        \quad\text{for all}\quad
        \triple_1,\triple_2\in\X,
    \end{equation*}
    where
    \begin{equation}
        \label{eq:dataterm:cliptot}
        C_{\mathrm{tot}}^{\mathrm{Lip}} = C_A^2(\eulipu L_S + \eulipx )M_J + C_A(\julipu L_S + \julipx) L_S M_e + C_AM_J\eulipx
    \end{equation}
    for constants
    $\eulipu , \eulipx$ (\cref{lem:eu-lipschitz}),
    $M_e$ (\cref{lem:ep-lipschitz}),
    $C_A$ (\cref{prop:adjoint-inverse}),
    $\julipu, \julipx$ (\cref{ass:dataterm:ju-lip}),
    $L_S$ (\cref{lemma:dataterm:ju-lip}), and
    $M_J$ (\cref{lem:data-bound}).
\end{theorem}

\begin{proof}Let $w_\triple \in W$ solve the reduced adjoint equation \eqref{eq:dataterm:reduced-adjoint}, i.e.,
    \[
        w_\triple = [e_u(\triple)^*]^{-1} \diffwrt{J}{u}(S_u(\triple)),
        \quad\text{where}\quad
        e_u(\triple)\defeq\diffwrt{e}{u}(S_u(\triple), \triple):U\to W^*.
    \]
    Then $w_\triple\diffwrt{e}{\triple}(S_u(\triple), \triple) = \diffwrt{J}{u}(S_u(\triple), \triple)S'(\triple)$ by \eqref{eq:dataterm:reduced-adjoint-reco}, so the claim amounts to showing the Lipschitz continuity of $\triple \mapsto w_\triple\diffwrt{e}{\triple}(S_u(\triple), \triple)$.

    Denote $e_\triple(\triplealt)\defeq \diffwrt{e}{\triple}(S_u(\triplealt),\triplealt) \in \linear(\X; W^*)$. We decompose
    \[
        \Delta \defeq w_{\triple_2}e_\triple(\triple_2) - w_{\triple_1}e_\triple(\triple_1)
        = (w_{\triple_2}-w_{\triple_1})e_\triple(\triple_2) + w_{\triple_1}(e_\triple(\triple_2)-e_\triple(\triple_1)) \in \X^*.
    \]
    By \cref{prop:adjoint-inverse},
    \[
        \norm{w_{\triple_i}}_U \le C_A\norm{\diffwrt{J}{u}(S_u(\triple_i))}_{U^*} \le C_AM_J.
    \]
    \Cref{lem:eu-lipschitz} gives $\norm{e_\triple(\triple_i)}_{\linear(\X;W^*)} \le M_e$ and
    \Cref{lem:ep-lipschitz} gives
    \[
        \norm{e_\triple(\triple_2)-e_\triple(\triple_1)}_{\linear(\X;W^*)} \le L_e\norm{\triple_2-\triple_1}_\X.
    \]
    For the difference, abbreviate $T_i \defeq e_u(\triple_i)^*: W\to U^*$. Then
    \begin{equation}
        \label{eq:difference}
        w_{\triple_2}-w_{\triple_1} = T_2^{-1}\bigl[\diffwrt{J}{u}(S_u(\triple_2))-\diffwrt{J}{u}(S_u(\triple_1))\bigr]
        + (T_2^{-1}-T_1^{-1})\diffwrt{J}{u}(S_u(\triple_1))
        \in W.
    \end{equation}
    By \cref{lemma:dataterm:ju-lip}, we have
    \begin{equation}
        \label{eq:term1}
        \norm{T_2^{-1}\bigl[\diffwrt{J}{u}(S_u(\triple_2))-\diffwrt{J}{u}(S_u(\triple_1))\bigr]}_{W} \le C_A\jfullulipx\norm{\triple_2-\triple_1}_\X.
    \end{equation}
    For the resolvent difference, since $u_i = S_u(\triple_i)$, \eqref{eq:lip-diffe} in
    \cref{lem:eu-lipschitz,lem:Su-lipschitz} give
    \[
        \begin{split}
            \norm{T_2-T_1}_{\linear(W;U^*)}
             &
            = \norm{e_u(S_u(\triple_2),\triple_2)^* - e_u(S_u(\triple_1),\triple_1)^*}
            \\
             &
            \le \eulipu \norm{S_u(\triple_2)-S_u(\triple_1)}_U + \eulipx \norm{\triple_2-\triple_1}_\X
            \le
            L_e \norm{\triple_2-\triple_1}_\X
        \end{split}
    \]
    for $L_e \defeq \eulipu L_S + \eulipx$.
    Since $\norm{T_i^{-1}}\le C_A$, the Neumann series yields
    \[
        \norm{T_2^{-1}-T_1^{-1}}_{\linear(U^*;U)} = \norm{T_2^{-1}(T_1-T_2)T_1^{-1}}_{\linear(U^*;U)}
        \le C_A\cdot L_e\cdot C_A \norm{\triple_2-\triple_1}_\X = C_A^2L_e\norm{\triple_2-\triple_1}_\X.
    \]
    Thus,
    \begin{equation}
        \label{eq:term2}
        \norm{(T_2^{-1}-T_1^{-1})\diffwrt{J}{u}(S_u(\triple_1))}_{W} \le C_A^2L_eM_J\norm{\triple_2-\triple_1}_\X.
    \end{equation}
    Substituting \eqref{eq:term1} and \eqref{eq:term2} into \eqref{eq:difference}, we obtain
    \begin{equation}
        \label{eq:w-lipschitz}
        \norm{w_{\triple_2}-w_{\triple_1}}_{W} \le (C_A\jfullulipx + C_A^2L_eM_J)\norm{\triple_2-\triple_1}_\X.
    \end{equation}
    Finally, combining \eqref{eq:w-lipschitz} with \cref{lem:ep-lipschitz,lemma:dataterm:ju-lip},
    and using $L_e = \eulipu L_S + \eulipx$ and $\jfullulipx = \julipu L_S + \julipx$ (\cref{lemma:dataterm:ju-lip}),
    we obtain for $C_{\mathrm{tot}}^{\mathrm{Lip}}$ as stated that
    \begin{align*}
        \norm{\Delta}_{\X^*}
         & \le \norm{w_{\triple_2}-w_{\triple_1}}_{W}\norm{e_\triple(\triple_2)}_{\linear(\X;W^*)}
        + \norm{w_{\triple_1}}_{W}\norm{e_\triple(\triple_2)-e_\triple(\triple_1)}_{\linear(\X;W^*)}
        \\
         & \le (C_A\jfullulipx + C_A^2L_eM_J)M_e\norm{\triple_2-\triple_1}_\X
        + C_AM_JL_e\norm{\triple_2-\triple_1}_\X
        \\
         & = C_{\mathrm{tot}}^{\mathrm{Lip}}\norm{\triple_2-\triple_1}_\X.
        \qedhere
    \end{align*}
\end{proof}

Combining \cref{ass:dataterm:ju-lip} regarding $\diffwrt{J}{x}$ with \cref{lem:Su-lipschitz,thm:wp-ep-lipschitz}, we immediately obtain:
\begin{corollary}[Full Lipschitz estimate]
    \label{cor:f-lipschitz-differentiable}
    Suppose \cref{ass:main,ass:dataterm:ju-lip} hold, and let $R>0$. Then
    on the bounded admissible set $\X_R$ (defined in \cref{lem:Su-lipschitz}),
    $F: \triple \mapsto J(S_u(\triple), \triple)$ is locally Lipschitz-differentiable:
    \[
        \norm{F'(\triple_2) - F'(\triple_1)}_{\X^*}
        \le \bigl(C_{\mathrm{tot}}^{\mathrm{Lip}}(R) + \jxlipu L_S(R) + \jxlipx\bigr) \norm{\triple_2 - \triple_1}_\X,
        \qquad \triple_1,\triple_2\in\X_R,
    \]
    where $C_{\mathrm{tot}}^{\mathrm{Lip}}(R)$ is defined in \eqref{eq:dataterm:cliptot}, $\jxlipu, \jxlipx$ arise from \cref{ass:dataterm:ju-lip}, and $L_S(R)$ from \cref{lem:Su-lipschitz}.
\end{corollary}

\begin{proof}Recalling \eqref{eq:dataterm:fprime}, we have
    \begin{equation}
        \label{eq:dataterm:fprime:repeat}
        F'(\triple) = \diffwrt{J}{u}(S_u(\triple),\triple) S_u'(\triple) + \diffwrt{J}{\triple}(S_u(\triple),\triple).
    \end{equation}
    Let $\triple_1,\triple_2\in\X_R$. For the second term, \cref{ass:dataterm:ju-lip,lem:Su-lipschitz} yield the following
    \[
        \begin{split}
            \norm{\diffwrt{J}{\triple}(S_u(\triple_2),\triple_2) - \diffwrt{J}{\triple}(S_u(\triple_1),\triple_1)}_{\X^*}
             &
            \le \jxlipu \norm{S_u(\triple_2) - S_u(\triple_1)}_U + \jxlipx \norm{\triple_2 - \triple_1}_\X
            \\
             &
            \le
            (\jxlipu L_S(R) + \jxlipx) \norm{\triple_2 - \triple_1}_\X,
        \end{split}
    \]
    where the first inequality uses the local Lipschitz estimate of
    \cref{lem:Su-lipschitz} on $\X_R$.
    Thus, using \cref{thm:wp-ep-lipschitz} for the first term of \eqref{eq:dataterm:fprime:repeat}, establishes the claim. Combining both terms establishes the claim on $\X_R$.
\end{proof}

\section{Optimisation with measures}
\label{sec:optimisation}

Now that we have developed our data term, and its derivatives, it is time to solve the inverse problem
We recall that we expect there to be a finite number of leaks, modelled by the measure $\mu$. To introduce this a priori information into our problem, we apply Radon-norm regularisation, considering the problem \eqref{eq:problem}, which has also the auxiliary unknowns $k$ and $c$: the diffusion and the convection fields. Considering reasonably short time intervals, we take the latter to be constant in space and time.

We use the specific choices \eqref{eq:problem:specific} of the data fidelity and regularisation term.
Since the data term $F= J \circ S_u$ is nonconvex, the existing theory of most conditional gradient methods does not extend to it.
In \cref{cor:f-lipschitz-differentiable} we have, however, proved that $F$ is Lipschitz-differentiable, which is enough to apply forward-backward type methods.

\subsection{Basic algorithm}

We apply the forward-backward methods of \cite{tuomov-pointsource,tuomov-unbalanced} with a Radon-norm-squared proximal penalty.
Recall that we write $\triple=(\mu,k,c)$, and similarly $\this\triple=(\this\mu,\this k,\this c)$.
For the basic non-sliding variant, an exact update for the measure $\nexxt\mu$ at iteration number $n \in \N$ would be determined by the surrogate problem
\begin{subequations}
    \label{eq:alg}
    \begin{align}
        \label{eq:alg:mu-step}
         & \nexxt\mu             \in \argmin_{\mu \in \Masses}~
        F(\thisx) + \iprod{\diffwrt{F}{\mu}(\thisx)}{\mu-\this\mu}
        + \alpha\norm{\mu}_{\Masses} + \delta_{\ge 0}(\mu)
        + \frac{1}{2\tau}\norm{\mu-\this\mu}_{\Masses}^2.
        \intertext{The auxiliary variables (convection and diffusion), would be updated by the standard Hilbert-space forward-backward step}
        \label{eq:alg:kc-step}
         & (\nexxt k, \nexxt c)  = \prox_{\sigma_{p} \Raux}((\this k, \this c) - \sigma_{p} \grad_{k,c} F(\thisx)).
    \end{align}
\end{subequations}
Here $\tau,\sigma_{p}>0$ are step length parameters (corresponding to the measure $\mu$ and the auxiliary parameters $k$ and $c$, respectively) that have to satisfy
\[
    \tau L_\mu < 1
    \quad\text{and}\quad \sigma_{p} L_{k, c} < 1,
\]
where $L_\mu$ and $L_{k,c}$ are the Lipschitz factors of $F'$ with respect to the indicated parameters.\footnote{The Lipschitz factor $L_{k,c}$ is denoted $L_z$ in \cite{tuomov-unbalanced}.}
In \cref{cor:f-lipschitz-differentiable} we have derived a joint factor $L_\mu = L_{k,c}$.

In practice, following \cite{tuomov-unbalanced}, inexact steps are taken for \eqref{eq:alg:mu-step} by solving for some iteration-dependent tolerances $\nexxt\epsilon>0$ satisfying $\sum_{n=0}^\infty \nexxt\epsilon < \infty$, the inexact optimality condition
\begin{subequations}
    \label{eq:alg:approx}
    \begin{gather}
        \label{eq:alg:approx:0}
        -\nexxt\epsilon \le \this v + \nexxt w + \norm{\nexxt\mu-\this\mu}_{\Masses}\nexxt\omega
        \le \nexxt\epsilon,
        \intertext{where $\this v \defeq \diffwrt{F}{\mu}(\this\triple)$ and $ \nexxt\omega, \nexxt w \in \Predual$ satisfy}
        \label{eq:alg:approx:vars}
        -1 \le \nexxt\omega \le 1,
        \
        \iprod{\nexxt \omega}{\nexxt\mu-\this\mu}=\norm{\nexxt\mu-\this\mu}_{\Masses},
        \
        \nexxt w \le \alpha,
        \
        \iprod{\nexxt w}{\nexxt\mu}=\alpha\norm{\nexxt\mu}_{\Masses}.
    \end{gather}
\end{subequations}
To solve these conditions, one finds an (approximate) minimiser $\xi \in \Omega$ of $\this v \in \Predual$,
and writing $\this\mu=\sum_{i=1}^{m_n} \beta_{n,i} \delta_{\xi_{n,i}}$,
forms $\nexxt\mu = \sum_{i=1}^{m_n} \beta_{n+1,i} \delta_{\xi_{n,i}}+ \beta_{n+1,m_n+1} \delta_\xi$
by solving a finite-dimensional optimisation for the escape rates $\{\beta_{n+1,i}\}_{i=1}^{m_n+1}$.
As this algorithm is a special case of the next algorithm (with $\theta=0$ and $\nexxt\gamma=0$), we defer the discussion of convergence until after its introduction.

\subsection{Sliding algorithm}

Our second algorithm applies ideas from optimal transport to first transport the Dirac masses of $\this\mu$ to new locations using a transport plan $\nexxt\gamma = \sum_{i=1}^{m_n} \eta_{n,i}\delta_{(x_{n,i}, y_{n,i})} \in \Plans$, where $\eta_{n,i}$ is the amount of mass transported from $x_{n,i}$ to the predicted new location $y_{n,i}$.
The transport or “sliding” of the locations is based on doing a gradient step on the locations with $v^n$.
Then the method applies the same insertion and weight optimisation steps as above at the transported measure.
These steps may have to be repeated several times to satisfy technical curvature and remainder conditions from \cite{tuomov-pointsource}.
As we will see in the numerical experiments, the transport or “source sliding” can significantly improve the performance of the algorithm.

To be precise, define the projections $\pi^i(\xi)=\xi_i$, as well as the push-forward measures $\pi_\#^i\mu$ satisfying $\pi_\#^i\mu(A)=\mu(\{\xi \in \Omega \mid \pi^i(\xi) \in A\})$.
A “sliding” version of \eqref{eq:alg} first replaces $\this\mu$ by $\this{\breve\mu}=\this\mu+(\pi_\#^1-\pi_\#^0)\nexxt\gamma$ for a transport plan $\nexxt\gamma\in\Plans$.
Then it applies \eqref{eq:alg} with $\this\mu$ replaced by $\this{\breve\mu}$ and $\this\triple$ replaced by $\this{\breve\triple}=(\this{\breve\mu},\this k,\this c)$.
That is, writing
\[
    \this v=\diffwrt{F}{\mu}(\this\triple),
    \quad\text{and}\quad
    \this{\breve v}=\diffwrt{F}{\mu}(\this{\breve\triple}),
\]
for some sliding parameter $\theta>0$, we first construct $\nexxt\gamma \in \Plans$ that satisfies
\begin{subequations}
    \label{eq:alg-sliding}
    \begin{align}
        \label{eq:alg-sliding:mu-prediction}
         & \newspatial = \spatial - \theta \tau \grad \thisv(\spatial)
        \quad\text{for}\quad \nexxt\gamma \text{-a.e. } (\spatial, \newspatial).
        \\
        \intertext{Then we approximately, as in \eqref{eq:alg:approx}, update}
        \label{eq:alg-sliding:mu-step}
         & \nexxt\mu \in \argmin_{\mu \in \Masses}~
        F(\this{\breve x}) + \iprod{\diffwrt{F}{\mu}(\this{\breve x})}{\mu-\this{\breve\mu}}
        + \alpha\norm{\mu}_{\Masses} + \delta_{\ge 0}(\mu)
        + \frac{1}{2\tau}\norm{\mu-\this{\breve \mu}}_{\Masses}^2,
        \shortintertext{and}
        \label{eq:alg-sliding:kc-step}
         & (\nexxt k, \nexxt c) \defeq \prox_{\sigma_p \Raux}((\this k, \this c) - \sigma_p \grad_{k,c} F(\this{\breve x})).
    \end{align}
\end{subequations}
This construction may have to be repeated several times until $\nexxt\gamma$ satisfies further restrictions. We typically
\begin{enumerate}[nosep]
    \item Start with the source marginal $\pi_\#^0\nexxt\gamma=\this\mu$, constructing the target marginal $\pi_\#^1\nexxt\gamma$ using \eqref{eq:alg-sliding:mu-prediction}.
          Then $\this{\breve\mu}$ is simply $\this\mu$ transported along the negative gradient $\thisv(\spatial)$.
    \item We reduce the mass transported by $\nexxt\gamma$, repeating  \cref{eq:alg-sliding:mu-step,eq:alg-sliding:kc-step} until further conditions given in \cite{tuomov-pointsource} are satisfied.
\end{enumerate}

Due to the non-convexity and the inclusion of the auxiliary unknowns $k$ and $c$ in our problem, the most relevant convergence result from \cite{tuomov-pointsource} is its Theorem 7.2.
Its conditions are listed in \cite[Theorem 7.1]{tuomov-pointsource}, and are verified as follows:
\begin{itemize}\item For the final squared Radon norm proximal term in \eqref{eq:alg-sliding:mu-step}, the convexity requirement (i) clearly holds, while (vi) and (vii) follow from \cite[Example 4.9]{tuomov-pointsource}.
    \item The smoothness condition (ii) on $F$ follows from \cref{cor:f-lipschitz-differentiable}. The Lipschitz factors $L$ and $L_z$ corresponding to the measure and the auxiliary unknowns, respectively, arise from the corollary, their relative scaling dependent upon the choice of the norm in $\X$. We can take $\ell=0$.
    \item The condition (iii) amounts to the aforementioned bound $\sum_{n=0}^\infty \nexxt\epsilon < \infty$ on the solution tolerances.
          Unchanged from the experiments of \cite{tuomov-pointsource}, we take $\epsilon^n = 0.5\tau\alpha/(1+0.2n)^{1.4}$.
    \item The curvature lower bound (iv) we do not in practise enforce: it holds in a region of local convexity of $F$ with $\ell_F=0$. It can also be enforced for any preselected $\ell_F>0$ by controlling $\nexxt\gamma$. In particular, it always holds for the non-sliding algorithm (due to $\nexxt\gamma=0$).
    \item The step length bounds of (v) we have just discussed, aside from the additional factor $\ell_{\grad v}$ that balances between permitting poor transport through a residual term, or eliminating it.
          We can and do take it as $1/(\tau\theta)$.
    \item The final (viii) is a lower bound on the objective function, which clearly holds.
\end{itemize}

For further details, we refer to \cite{tuomov-pointsource}.

\section{Aspects of numerical implementation}
\label{sec:implementation}

In this section we describe aspects of numerical implementation of the convection--diffusion equation \eqref{eq:parabolic:convection-diffusion}, as well as its adjoint.
We first describe the implementation of the laser observation operator $A$ in \cref{sec:implementation:laser}.
We then apply the finite element method (FEM) and forward Euler steps to discretize \eqref{eq:parabolic:convection-diffusion} and obtain the corresponding discrete system (\cref{sec:discrete-equation}).
Finally, we describe some details of the optimisation algorithm in \cref{sec:implementation:algorithm}.
Further details are available in our software implementation \cite{dangvalkonen2026implementation}.

We introduce a temporal grid $0 = t_0 < t_1 < \cdots < t_{N_T} = T$ with time step lengths  $\Delta_i \defeq t_{i+1} - t_{i}$, for $i = 0, \dots, N_T-1$.
For simplicity of presentation, we assume that both the laser observation times and time discretisation agree with this grid.

\subsection{Discretisation of the convection--diffusion equation}
\label{sec:discrete-equation}

\paragraph{Discretisation of the concentration (state)}
\label{sec:discrete-equation:forward}

We write $I_i = [t_{i}, t_{i+1}]$, for $i = 0, \dots, N_T-1$ with $t_0 = 0$, $t_{N_T} = T$.
For space discretisation, we choose a finite element basis $\{\phi_m\}_{m=1}^{N_\phi}$ that forms $\mathbb{L}^2 \subset L^2(\Omega)$.
Then we define the space of discretised concentrations by
\[
    \mathbb{U}
    \defeq
    \left\{
    u \in L^2(I \times \Omega)
    \,\middle|\,
    \begin{array}{l}
        u(\freevar, t_i) \in \mathbb{L}^2 \text{ for all } 0 \le i \le N_T,
        \\
        u(x, \freevar)|I_i = \text{constant} \text{ for all } x \in \Omega, \quad i = 0, \dots, N_T-1
    \end{array}
    \right\}.
\]
Denoting the indicator function of $I_i$ by $\chi_i$, each $u \in \mathbb{U}$ admits for some  coefficient vectors $\mathbf{u}_i=(u_1^i, \ldots, u_{N_\phi}^i) \in \R^{N_\phi}$, ($i=0,\ldots,N_T-1$) the representation
\[
    u(\spatial, t)
    = \sum_{i=0}^{N_T-1} \sum_{m=1}^{N_\phi}
    u_m^i\, \chi_i(t)\, \phi_m(\spatial).
\]

We discretize the continuous problem \eqref{eq:parabolic:convection-diffusion} in space and time.
Let $a(\cdot,\cdot)$ be the bilinear form associated with the operator $-\Delta + c\cdot\nabla$, that is
\[
    a(u, v; \triple) = \int_\Omega k\, \nabla u \cdot \nabla v + (c \cdot \nabla u)\, v \dx.
\]

First, the forward Euler scheme on each time interval $I_i=[t_i,t_{i+1}]$ yields
\begin{equation}
    \label{eq:discrete-equation:forward-time}
    \adaptiprod{\frac{u(\freevar, t_{i+1})- u(\freevar, t_i)}{\Delta_i}}{v}
    + a(u(\freevar, t_i), v; \triple)
    = \iprod{\mu}{v} , \quad u(\freevar, t_0) = u_0,
\end{equation}
for all $ v \in \mathbb{V}_m \subset \{v_m \in W^{1,p}(\Omega) : v_m|_{\Gamma_1} = 0\} $. No boundary integral of $g$ appears, because $v|_{\Gamma_1}=0$.
Taking $u, v \in \mathbb{U}$, we can write this in terms of coefficients as
\[
    \sum_{m,r=1}^{N_\phi}
    \left(
    \frac{u_m^{i+1}- u_m^i}{\Delta_i} \iprod{\phi_m}{\phi_r} v_r^i
    + u_m^i a(\phi_m, \phi_r; \triple) v_r^i
    \right)
    = \sum_{r=1}^{N_\phi} \iprod{\mu}{\phi_r} v_r^i    \quad\text{for all}\quad
    \mathbf{v}^i \in \R^{N_\phi}
\]
with $u_{m,0} = [u_0]_m$ for all $m=1,\ldots,N_\phi$. That is,
\begin{equation}
    \label{eq:discrete-equation:forward-full}
    \mathbf{M} \mathbf{u}^{i+1}
    =
    \left(\mathbf{M} - \Delta_i \mathbf{K}(\triple)\right)\mathbf{u}^i
    + \Delta_i \mathbf{f}^i
    \quad\text{for all}\quad
    i=0,\dots,N_T-1,
\end{equation}
where
\[
    \mathbf{f}^i = \bigl(\iprod{\mu}{\phi_r}  \bigr)_r,
    \quad
    \mathbf{M} = \bigl(\iprod{\phi_m}{\phi_r}\bigr)_{m,r}
    \quad
    \mathbf{K}(\triple)=\bigl(a(\phi_m,\phi_r; \triple)\bigr)_{m,r}
    \quad\text{with}\quad
    m,r \in \{1,\ldots,N_\phi\}.
\]

For $i=0,\ldots,N_T-1$, let
\[
    \mathbf B^i(\triple)
    \defeq
    \mathbf M-\Delta_i\mathbf K(\triple),
\]
where $\mathbf M$ and $\mathbf K^i(\triple)$ are the mass and
stiffness--convection matrices in
\eqref{eq:discrete-equation:forward-full}. The forward recurrence for
computing $\mathbf u^{i+1}$ from $\mathbf u^i$ reads
\begin{equation}
\label{eq:discrete-equation:forward-residual}
    \mathbf M\mathbf u^{i+1}
    =
    \mathbf B^i(\triple)\mathbf u^i
    +
    \Delta_i\mathbf f^i(\triple),
    \qquad
    i=0,\ldots,N_T-1.
\end{equation}
\paragraph{Stability.}
Forward Euler is conditionally stable. Let $k_0\le k\le k_\infty$,
$\|c\|_\infty<\infty$, and let $C_{m,r}$ satisfy
$\|\nabla v\|_{L^2}\le C_{m,r}h^{-1}\|v\|_{L^2}$ on $\mathbb V_m$.
If
\begin{equation}\label{eq:cfl}
  \Delta_i\bigl(k_\infty C_{m,r}m^{-1}+\|c\|_\infty\bigr)^2\le k_0,
\end{equation}
then, in the homogeneous case,
$
\|u_m^{i+1}\|_{L^2}^2
\le
\left(1+\frac{\Delta_i\|c\|_\infty^2}{k_0}\right)
\|u_m^i\|_{L^2}^2.
$
Thus $\Delta_i=\mathcal O(h^2)$ is required. With a consistent $\mathbf M$ each step needs a mass-matrix solve; replacing $\mathbf M$ by the lumped matrix $\mathbf M_L$ (row sums) makes the step explicit and is
admissible under \eqref{eq:cfl}.
If lumping is used, $\mathbf M_L$ should also be used in the adjoint and gradient. For convection-dominated problems, i.e., $\|c\|_\infty m/k_0\gtrsim1$, stabilisation such as streamline upwind may be needed and must be included in the discrete adjoint and in the gradient calculation \cite{Brooks1982}.  General stability and convergence analysis for finite-element approximations of time-dependent PDEs is discussed in \cite{quarteroni1994numerical}, while robust methods for convection-dominated diffusion are surveyed in \cite{RoosStynesTobiska2008}.

\paragraph{Discretisation of the adjoint variable and equation}
The discrete adjoint system arises from the same space--time Galerkin discretization as the state equation and is therefore adjoint consistent.
The adjoint variable $w \in U_0\defeq \{ v \in U : v|_{\Sigma_1} = 0 \}$ is discretised in the same space $\mathbb{U}$ as the concentration.
The weak adjoint equation $e_u(\triple)^*w_\triple = \diffwrt{J}{u}(S_u(\triple), \triple)$.
Let \(a(\cdot,\cdot;\triple)\) denote the bilinear form of the forward
problem. We define the adjoint bilinear form by
\[
    \tilde a(w,v;\triple)
    \defeq
    a(v,w;\triple).
\]
This reads
\[
    \tilde a(w,v;\triple)
    =
    \int_\Omega
    k\,\nabla w\cdot\nabla v
    -
    (c\,w)\cdot\nabla v
    \,\d\xi,
\]
subject to the boundary conditions encoded in \(U_0\).
The discrete adjoint is solved backward in time. In weak form, it is given by
\[
    \adaptiprod{\frac{w(\freevar, t_i)-w(\freevar, t_{i+1})}{\Delta_i}}{v}
    -
    \tilde a(w(\freevar, t_{i+1}),v;\triple)
    =
    \diffwrt{J}{u(\freevar, t_i)}(S_u(\triple),\triple)\,v, \quad\text{for all}\quad
    v \in  U_0.
\]
for all \(v\in\mathbb U\) and \(i=0,\ldots,N_T-1\).
Writing
$
    w^i
    =
    \sum_{m=1}^{N_\phi}
    \mathbf w_m^i\phi_m,
$
we obtain
\begin{equation}
\label{eq:discrete-adjoint}
    \mathbf M\mathbf w^i
    =
    \left(
        \mathbf M+\Delta_i\mathbf L(\triple)
    \right)\mathbf w^{i+1}
    +
    \Delta_i\mathbf M\mathbf b^i(\triple)
    \qquad
    \text{for }i=0,\ldots,N_T-1,
\end{equation}
where the matrix and vector
\begin{equation}
    \label{eq:discrete-equation:mat-vec-lb}
    \mathbf{L}(\triple) = \bigl(\tilde{a}(\phi_m,\phi_r; \triple)\bigr)_{m,r} = \mathbf K(\triple)^\top,
    \quad\text{and}\quad
  \mathbf M  \mathbf{b}^i(\triple)
    = \bigl(\diffwrt{J}{u^i_r}(S_u(\triple), \triple)\bigr)_r
    \quad\text{with}\quad
    m,r \in \{1,\ldots,N_\phi\}.
\end{equation}
We will refine the expression for $\mathbf{b}^i(\triple)$ in \cref{sec:implementation:laser}, after fixing $J$ and the observation operator $A$ within it.

The adjoint variable is thus the discrete adjoint of the actual forward operator. It vanishes on
$\Gamma_1$ (it lives in $\mathbb V_m$), and its boundary treatment follows from the lift in
\eqref{eq:discrete-equation:forward-residual}. Integrating by parts,
$\tilde a(w,v)=\int_\Omega k\nabla w\cdot\nabla v-(c\cdot\nabla w)v-(\nabla\!\cdot c)wv\dx+\int_{\Gamma_2}(c\cdot n)wv\dd\sigma$,
which is the form associated with $-\nabla\!\cdot(k\nabla\,\cdot)-\nabla\!\cdot(c\,\cdot)$; it reduces to
$-\Delta-c\cdot\nabla$ only for $\nabla\!\cdot c=0$ and $c\cdot n=0$ on $\Gamma_2$.
The iteration matrix of \eqref{eq:discrete-adjoint}, $\mathbf M^{-1}\mathbf B^\top$, is similar to
$(\mathbf M^{-1}\mathbf B)^\top$, so \eqref{eq:cfl} guarantees stability of the adjoint as well.

\paragraph{Discretisation of the parameters}

We now consider the discretisation of the parameters $\triple=(\mu,k,c)$.
Practically, the algorithms of \cref{sec:optimisation} ensure that on each iteration $n$, the measure $\this\mu=\sum_{i=0}^{m_n} \beta_{n,i} \delta_{\xi_{n,i}}$ is a weighted sum of Dirac measures.
We do not enforce the source points $\xi_{n,i}$ to lie on any specific grid: aside from the discretisation of $u$, and of $\diffwrt{F}{\mu}(\triple)$, our algorithm is partially grid-free.

In our numerical experiments, the convection and diffusion field $k$ and $c$ are scalars, $c(x) \equiv (c_1, c_2) \in \R^2$ and $k(x) \equiv k_0 \in \R$.

\subsection{Laser observation operator}
\label{sec:implementation:laser}

The observation model for gas emission detection is based on path-integrated measurements obtained via Laser Diode Spectroscopy (LDS). Each laser beam provides a path-averaged signal proportional to the line integral of the gas concentration along its trajectory.
At each time step $t_i$, $i=0,\ldots,N_T-1$, the concentration field $u(\spatial,t)$ is observed through $N_\ell$ laser beams $\{ \gamma_\ell \}_{\ell=1}^{N_\ell}$ (see \cref{fig:lasers-and-mirrors}) reflected back to their origin from mirrors.
This leads to the linear observation model
\[
    A: L^2(\Omega_T) \to \R^{N_\ell \times N_T},
    \quad
    [Au]_{\ell,i} = \int_{\gamma_\ell} u(\xi, t_i) \dd s,
\]
where $\ell=1,\ldots,N_\ell$ runs over the laser beams, and $i=0,\ldots,N_T-1$ over the observation times $t_0, \ldots,t_{N_T-1} \in [0, T)$.
Thus the path-averaged LDS measurements are
\[
    b = A\hat u + \nu,
\]
for the ground truth concentration $\hat u \in L^2(\Omega_T)$, and (Gaussian) noise $\nu$.

For $u \in \mathbb{U}$ with the coefficient vectors $\mathbf{u}^1,\ldots,\mathbf{u}^{N_T} \in \R^{N_\phi}$ as in \cref{sec:discrete-equation}, we get
\[
    [Au]_{\ell,i}
    = \int_{\gamma_\ell} u(\spatial,t_i) \dd s
    =
    \sum_{m=1}^{N_\phi} u_m^i
    \int_{\gamma_\ell} \phi_m(\spatial)\dd s
    =
    \begin{pmatrix}
        \int_{\gamma_\ell} \phi_1(\spatial)\dd s, & \ldots, & \int_{\gamma_\ell} \phi_m(\spatial)\dd s
    \end{pmatrix}
    \mathbf{u}^i.
\]
For each observation time index $i=1,\ldots,N_T-1$, the ideal discretisation of $u \mapsto [Au]_{\freevar, i} \in \R^{N_\ell}$ is, therefore,
\[
    \obs \in \R^{N_\ell \times N_\phi}
    \quad\text{with}\quad
    \obs_{\ell,m} = \int_{\gamma_\ell} \phi_m(\spatial)\dd s.
\]
We do this integration exactly using midpoint quadrature.

Taking $J(u, \triple)=\frac{1}{2}\norm{Au-z_d}_2^2$, we now have
\[
    J(u, \triple) = \sum_{i=0}^{N_T-1} \frac{1}{2}\norm{\mathbf{A}\mathbf{u}^i-z_d^i}_2^2
    \quad\text{when}\quad
    u \in \mathbb{U}.
\]
Thus, $\mathbf{b}^i(\triple)$ defined in \eqref{eq:discrete-equation:mat-vec-lb} can be written
\[
    \mathbf{b}^i(\triple) = (\mathbf{A}^\top(\mathbf{A}\mathbf{u}^i-z_d^i))_i, \quad i=0,\dots,N_T-1.
\]
Further details are available in our software implementation \cite{dangvalkonen2026implementation}.

\subsection{The optimisation algorithm}
\label{sec:implementation:algorithm}

The implementation of the optimisation algorithm is described in \cite{tuomov-pointsource,tuomov-unbalanced}.
Our present application shares most of the codebase, with the changes related to the implementation and treatment of the forward operator, which in \cite{tuomov-pointsource,tuomov-unbalanced} was a simple convolution.
In particular, $v= \diffwrt{F}{\mu}(x)$ now belongs to a P2 (second order polynomial) finite element subspace of $\Predual$. Therefore, we do not need to use a branch-and-bound algorithm to find---what were---grid-free minimisers of this function. We can, instead, solve the minima exactly by minimising the function on each element, and taking a best elementwise minimiser as the global minimiser.

\section{Numerical results}
\label{sec:numerical}

We now describe our numerical experiments on leak detection.
The details of the numerical implementation of the model \cref{eq:problem,eq:problem:specific,eq:parabolic:convection-diffusion}, which is available in \cite{dangvalkonen2026implementation}, are as in \cref{sec:implementation}.
We describe the specific experimental setup and algorithm parametrisation in \cref{sec:numerical:setup,sec:numerical:alg}, and then discuss the results in \cref{sec:numerical:res}.

\subsection{Experimental setup}
\label{sec:numerical:setup}

\begin{figure}\centering
    \begin{tikzpicture}[scale=5]
        \pgfmathsetmacro{\mirrorsperside}{10}
        \foreach \x in {0.1, 0.4} {
                \foreach \y in {0.1, 0.4} {
                        \node[draw,draw=red!10!white,fill=red,star,star points=10,star point ratio=4,scale=0.3] at (\x, \y) {};
                    }
            }

        \foreach \i in {1,...,\mirrorsperside} {
                \pgfmathsetmacro{\c}{0.5*\i/(\mirrorsperside+1)}
                \pgfmathsetmacro{\s}{\c-0.01}
                \pgfmathsetmacro{\e}{\c+0.01}

                \foreach \a in {0, 0.5} {
                        \draw[line width=1.5pt,color=SkyBlue] (\s, \a)--(\e, \a);
                        \draw[line width=1.5pt,color=SkyBlue] (\a, \s)--(\a, \e);

                        \foreach \x in {0.1, 0.4} {
                                \foreach \y in {0.1, 0.4} {
                                        \draw[line width=0.5pt,color=red,opacity=0.07] (\x, \y)--(\c, \a);
                                        \draw[line width=0.5pt,color=red,opacity=0.07] (\x, \y)--(\a, \c);
                                    }
                            }
                    }
            }

    \end{tikzpicture}
    \caption{Illustration of the laser and mirror placements. The thick blue lines on the boundary indicate the mirrors, the red stars the laser sources, and the dimmed thin red lines the beams.}
    \label{fig:lasers-and-mirrors}
\end{figure}

We divide the time interval $I=[0, 1]$ into 50 uniform segments.
The domain is $\Omega=[0, 0.5]^2$ with a uniform grid of $32 \times 32$ nodes.
The laser sources are at $\{(0.1, 0.1), (0.1, 0.4), (0.4, 0.1), (0.4, 0.4)\}$, and each edge of the domain has 10 uniformly spaced mirrors, as illustrated in \cref{fig:lasers-and-mirrors}.
Further details can be found in our numerical implementation \cite{dangvalkonen2026implementation}.
This gives $4 \cdot 10 \cdot 4 = 160$ data points for each time step, altogether $160 \cdot 50 = 8000$ data points, to be compared with altogether $32 \cdot 32 \cdot 50 = 51200$ nodal values for the concentration field.
For simulation of the data, twice as many time steps, as well as nodes in each coordinate direction, are used.

We perform two experiments with different atmospheric conditions and leak sources.
In experiment \#1, the true convection field is $\hat c \equiv 0.5(\cos(30^\circ), \sin(30^\circ))$, and the true diffusion field $\hat k \equiv 0.01$ with the true source measure
\[
    \hat\mu = 0.08\,\delta_{(0.1,0.3)} + 0.05\,\delta_{(0.4, 0.25)} + 0.06 \,\delta_{(0.25, 0.13)}
\]
In experiment \#2, the true convection field is $\hat c \equiv 0.1(\cos(120^\circ), \sin(120^\circ))$, and the true diffusion field $\hat k \equiv 0.02$ with the true leaks modelled by the measure
\[
    \hat\mu = 0.15\,\delta_{(0.2,0.3)} + 0.04\,\delta_{(0.4, 0.1)}.
\]
(The latter leak coincides with one of the laser sources).
The measurements $\tilde k$ and $\tilde c_i$ both have $2\%$ Gaussian noise: the PDE is much more sensitive to the diffusion parameter $k_0$, and cannot tolerate higher levels of noise.
The laser line-of-sight measurements $b$ have 1\% Gaussian noise.

The boundary conditions are $u=g \defeq 0$ on $\Gamma_1 = \{0,0.5\} \times [0,0.5]$ and $\partial u/\partial n=0$ on $\Gamma_2=(0,0.5) \times \{0, 0.5\}$. The initial concentration is $u_0=0$.

For the regularisation parameter of the measure, by trial and error, we take $\alpha=2 \cdot 10^{-3}$, while we use a combination of box constraints and quadratic fitting with weighting discovered by trial-and-error in
\[
    \Raux(k, c) = \delta_{[0.001, 1]}(k_0) +\frac{60}{2}(k_0-\tilde k)^2
    + \sum_{i=1,2}\left(
    \delta_{[-1, 1]}(c_i) + \frac{0.06}{2}(c_i - \tilde c_i)^2
    \right),
\]
where we recall that we take the convection and diffusion fields $c \equiv (c_1, c_2) \in \R^2$ and $k \equiv k_0 \in \R$ to be constant throughout the domain.
Their respective measurements are $\tilde k$, $\tilde c_1$, and $\tilde c_2$.

\subsection{Algorithm parametrisation}
\label{sec:numerical:alg}

We take as the primal step length parameter $\tau=0.02$ and $\sigma_p=0.0001$, based on manual fine-tuning.
The sliding step length parameter is $\theta=0.07/\tau$, likewise based on manual fine-tuning.
We use a merging strategy to reduce the number of sources in the measure iterates $\mu^n = \sum_{i=1}^{m_n} \beta_{n,i}\delta_{\xi_{n,i}}$: every 10 iterations, we replace any two sources $i \ne j$ with $\norm{\xi_{n,j}-\xi_{n,i}} \le 0.1$ with $(\beta_{n,j}+\beta_{n,i})\delta_{\xi_{n,j}}$ if this does not increase the objective function value, within a decreasing tolerance. This helps to keep the number of sources controlled, and, therefore, avoid the per-iteration computational demands growing.

As initial iterates, we take measured values $c^0=\tilde c$ and $k^0=\tilde k$ of the convection and diffusion parameters. A priori, we assume there to be no leaks, taking as initial iterate of the leak-modelling measure $\mu^0 = 0$.

\subsection{Numerical results}
\label{sec:numerical:res}

\begin{figure}[t]
    \subcaptionbox{True concentration $\hat u$ and measure $\hat\mu$}{%
        \includegraphics[width=\linewidth]{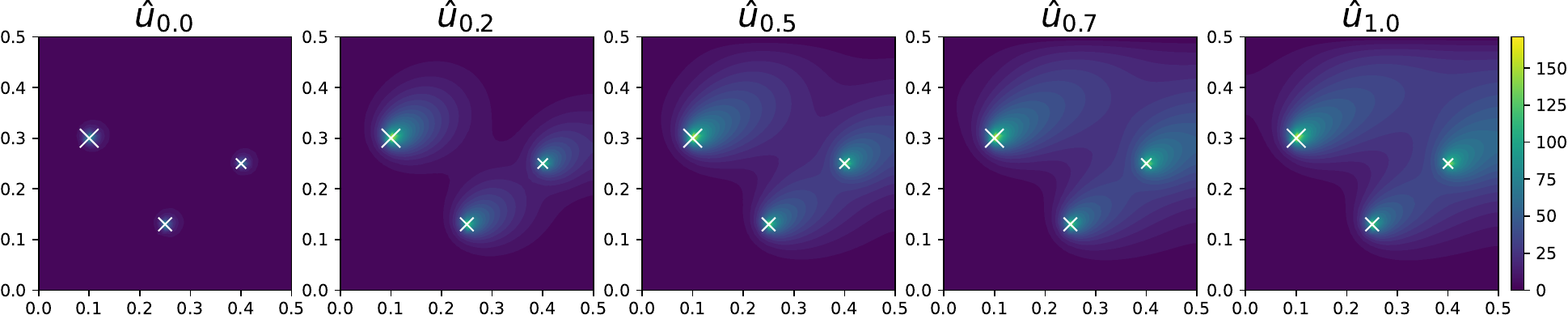}
    }
    \subcaptionbox{Sliding method reconstruction of concentration $u$ and measure $\mu$}{%
        \includegraphics[width=\linewidth]{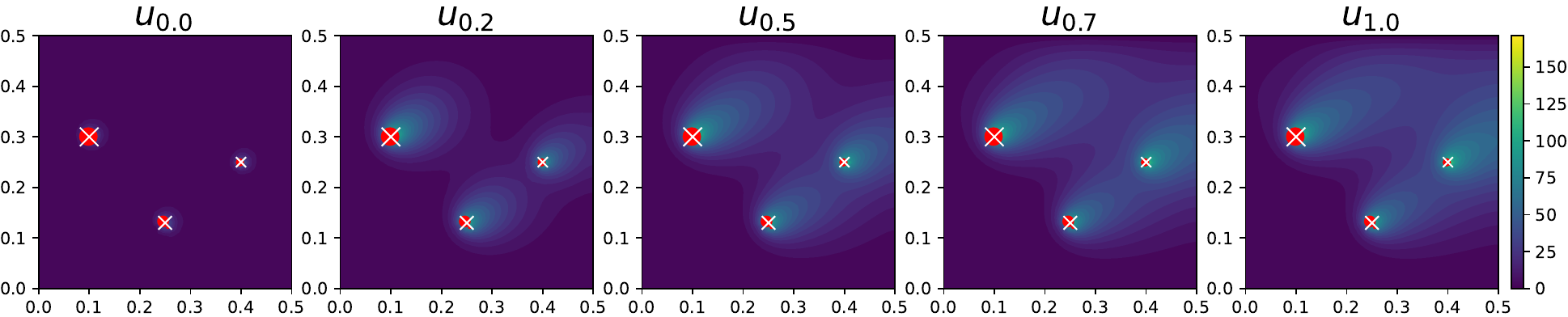}
    }
    \subcaptionbox{Basic method reconstruction of concentration $u$ and measure $\mu$}{%
        \includegraphics[width=\linewidth]{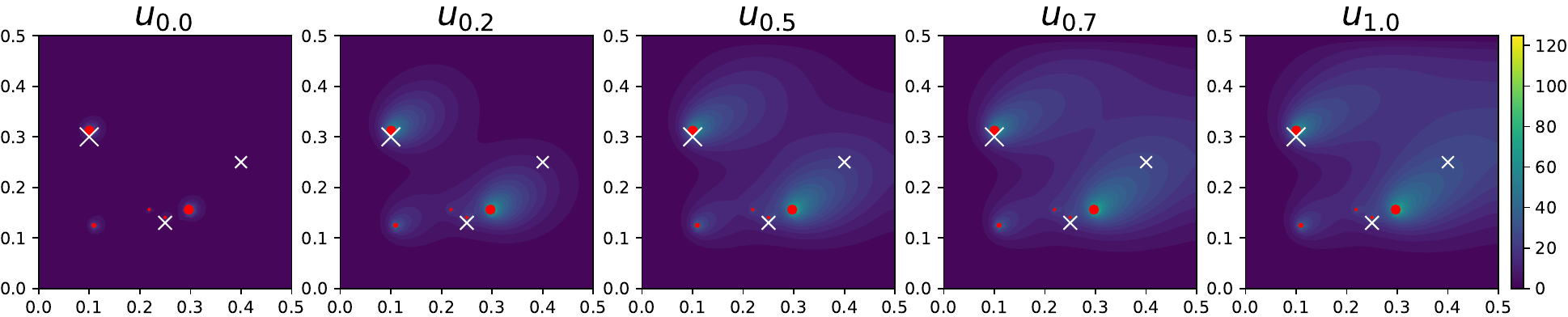}
    }
    \caption{%
        Reconstructions, experiment \#1.
        Time runs from left to right, with the reconstructions shown for $t=0.0,0.2,0.5,0.7,1.0$.
        The shading indicates the concentration field $u$.
        The red dots indicate the reconstructed measure, and the white crosses the ground true measure.
        The size of the dots corresponds to the escape rate.
    }
    \label{fig:recos1}
\end{figure}

\begin{figure}[t]
    \subcaptionbox{True concentration $\hat u$ and measure $\hat\mu$}{%
        \includegraphics[width=\linewidth]{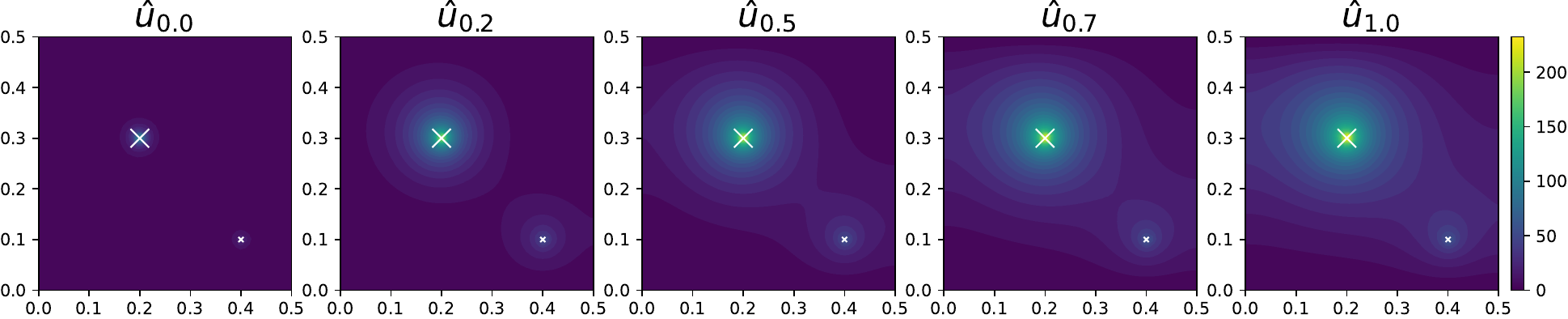}
    }
    \subcaptionbox{Sliding method reconstruction of concentration $u$ and measure $\mu$}{%
        \includegraphics[width=\linewidth]{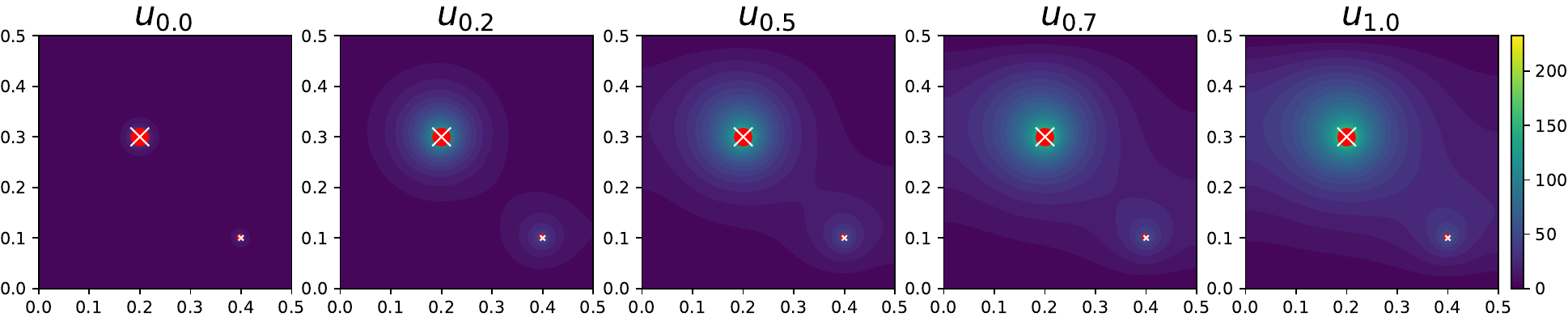}
    }
    \subcaptionbox{Basic method reconstruction of concentration $u$ and measure $\mu$}{%
        \includegraphics[width=\linewidth]{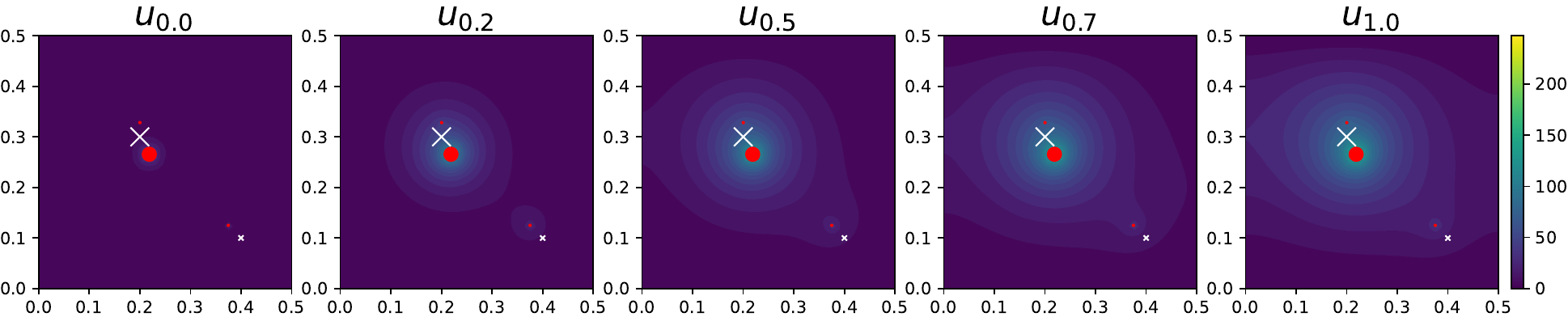}
    }
    \caption{%
        Reconstructions, experiment \#2.
        Time runs from left to right, with the reconstructions shown for $t=0.0,0.2,0.5,0.7,1.0$.
        The shading indicates the concentration field $u$.
        The red dots indicate the reconstructed measure, and the white crosses the ground true measure.
        The size of the dots corresponds to the escape rate.
    }
    \label{fig:recos2}
\end{figure}

\newlength{\figheight}
\setlength{\figheight}{0.35\textwidth}
\newlength{\twofigheight}
\setlength{\twofigheight}{0.6\figheight}
\newlength{\figwidth}
\setlength{\figwidth}{0.5\textwidth}
\newlength{\dataplotheight}
\setlength{\dataplotheight}{0.3\linewidth}
\newlength{\twodkernelplotheight}
\setlength{\twodkernelplotheight}{\figheight}

\pgfplotsset{
    ylabel near ticks,
    compat = 1.18,
    scale only axis = false,
    scale mode = stretch to fill,
    every axis/.append style = {
            trim axis left,
            trim axis right,
        },
    plot coordinates/math parser = false,
    every axis label/.append style = {font = \scriptsize},
    every tick label/.append style = {font = \scriptsize},
    x tick label style = {
            /pgf/number format/fixed,
            /pgf/number format/set thousands separator={\,}
        },
    legend style = {
            inner sep = 0pt,
            outer xsep = 5pt,
            outer ysep = 0pt,
            legend cell align = left,
            align = left,
            draw = none,
            fill = none,
            font = \scriptsize,
        },
    width = \linewidth,
    height = \figheight,
    scaled x ticks = false,
    xminorticks = true,
    unbounded coords = jump,
    axis x line* = bottom,
    axis y line* = left,
    yminorticks = true,
    ylabel style = {
            at = {(yticklabel* cs:0.5)},
            anchor = center,
            yshift = 7ex,
        },
    notfirst/.style = {skip coords between index={0}{1}},
    log x ticks with fixed point/.style={
            xticklabel={
                    \pgfkeys{/pgf/fpu=true}
                    \pgfmathparse{exp(\tick)}%
                    \pgfmathprintnumber[fixed relative, precision=3]{\pgfmathresult}
                    \pgfkeys{/pgf/fpu=false}
                },
        },
    log y ticks with fixed point/.style={
            yticklabel={
                    \pgfkeys{/pgf/fpu=true}
                    \pgfmathparse{exp(\tick)}%
                    \pgfmathprintnumber[fixed relative, precision=3]{\pgfmathresult}
                    \pgfkeys{/pgf/fpu=false}
                },
        },
}

\ExplSyntaxOn
\NewDocumentCommand{\addplotnamedtable}{O{}m}{%
    \use:e { \exp_not:n { \addplot table~[#1] } { \exp_not:c { #2 } } }
}
\NewDocumentCommand{\addplotnamedtableX}{O{}mm}{%
    \use:e { \exp_not:n { \addplot[#1] table~[#2] } { \exp_not:c { #3 } } }
}
\NewDocumentCommand{\pgfplotstablereadnamed}{mm}{%
    \use:e { \exp_not:n { \pgfplotstableread{#1} } { \exp_not:c { #2 } } }
}
\ExplSyntaxOff

\catcode`\_=12
\def\radonFBName{basic}\def\radonFBFilename{radon_fb}
\def\sFBradonName{sliding}\def\sFBradonFilename{radon_sliding_fb}
\catcode`\_=8

\newif\ifallalgs
\allalgstrue

\pgfkeys{/matrixdim/.initial = 16}

\pgfplotsset{
    solid mark/.style = {
            every mark/.append style = {solid},
        }
}

\pgfplotstableset{
    create on use/inner_iters_mean/.style={
            create col/expr={\thisrow{inner_iters} / \thisrow{this_iters}},
        }
}

\pgfplotscreateplotcyclelist{reconstructions}{
    {color = Set2-H, line width = 1pt, mark = *, mark size = 2pt},
    {color = Set2-A, line width = 1pt, mark = square*},
    {color = Set2-B, line width = 1.5pt, mark = o, mark size = 2.5pt, densely dotted, solid mark},
    {color = Set2-C, line width = 1pt, mark = triangle, mark size = 2.5pt},
    {color = Set2-D, line width = 1.5pt, mark = o, mark size = 3.5pt, densely dashdotted, solid mark},
    {color = Set2-H, line width = 1pt, mark = pentagon, mark size = 5pt, densely dashed, solid mark},
    {color = Set2-A, line width = 1pt, mark = square, mark size = 2pt, loosely dashed},
}

\colorlet{regionradonFB}{Set2-A}
\colorlet{regionsFBradon}{Set2-B}
\colorlet{regionradonFBAlt}{Set2-C}
\colorlet{regionsFBradonAlt}{Set2-D}

\def\ignorelegendentry#1{}

\pgfplotsset{
    reconstructionplot/.style = {
            reverse legend,
            cycle list name = reconstructions,
            ycomb,
            mark size = 3pt,
            legend columns = 2,
        },
    performanceplot/.style = {
            no markers,
            cycle list shift = 1,
            cycle list name = reconstructions,
        },
    valueplot/.style = {
        },
    ignore legend/.style={
            every axis legend/.code={\let\addlegendentry\ignorelegendentry}
        },
}

\newcommand{\findminmax}[5]{
    \def\tempa{#3}
    \def\tempb{true}
    \ifx\tempa\tempb
        \pgfplotstablegetelem{0}{#2}\of{#1}%
        \pgfmathparse{\pgfplotsretval}
        \global\let#5\pgfmathresult
        \pgfplotstablegetelem{0}{#2}\of{#1}%
        \pgfmathparse{\pgfplotsretval}
        \global\let#4\pgfmathresult
    \fi
    \pgfplotstablegetrowsof#1
    \pgfmathparse{\pgfplotsretval-1}
    \foreach \i in {0,...,\pgfmathresult}{
            \pgfplotstablegetelem{\i}{#2}\of{#1}%
            \pgfmathparse{max(#5, \pgfplotsretval)}
            \global\let#5\pgfmathresult
            \pgfplotstablegetelem{\i}{#2}\of{#1}%
            \pgfmathparse{min(#4, \pgfplotsretval)}
            \global\let#4\pgfmathresult
        }
}

\def\maxiter{10000}

\def\doplotr#1#2#3#4#5#6#7#8#9{
    \addplotnamedtableX [forget plot, name path={#6#1#2}, draw=none]
    {x = #9, y = #6}{#1Performance#2};
    \addplotnamedtableX [forget plot, name path={#7#1#2}, draw=none]
    {x = #9, y = #7}{#1Performance#2};
    \addplot[fill=region#1#2, fill opacity=0.1, forget plot] fill between[of={#6#1#2} and {#7#1#2}];
    \addplotnamedtableX [forget plot, name path={#4#1#2}, draw=none]
    {x = #9, y = #4}{#1Performance#2};
    \addplotnamedtableX [forget plot, name path={#5#1#2}, draw=none]
    {x = #9, y = #5}{#1Performance#2};
    \addplot[fill=region#1#2, fill opacity=0.3, forget plot] fill between[of={#4#1#2} and {#5#1#2}];
    \addplotnamedtable [x = #9, y = #3]{#1Performance#2};
    \addlegendentry{\csname#1Name\endcsname~ / #8}
}
\def\doplotraux#1#2#3#4#5#6#7#8#9{
    \addplotnamedtableX [forget plot, name path={#6#1#2}, draw=none]
    {x = #9, y = #6}{#1Aux#2};
    \addplotnamedtableX [forget plot, name path={#7#1#2}, draw=none]
    {x = #9, y = #7}{#1Aux#2};
    \addplot[fill=region#1#2, fill opacity=0.1, forget plot] fill between[of={#6#1#2} and {#7#1#2}];

    \addplotnamedtableX [forget plot, name path={#4#1#2}, draw=none]
    {x = #9, y = #4}{#1Aux#2};
    \addplotnamedtableX [forget plot, name path={#5#1#2}, draw=none]
    {x = #9, y = #5}{#1Aux#2};
    \addplot[fill=region#1#2, fill opacity=0.3, forget plot] fill between[of={#4#1#2} and {#5#1#2}];

    \addplotnamedtable [x = #9, y = #3]{#1Aux#2};
}

\newif\ifsummary
\summaryfalse

\def\doplot#1#2#3#4#5{%
    \ifsummary\doplotr{#1}{#2}{#4_mean}{#4_confmin}{#4_confmax}{#4_min}{#4_max}{#5}{#3}
    \else\addplotnamedtable [x = #3, y = #4]{#1Performance#2};
    \fi
    \addlegendentry{\csname#1Name\endcsname~ / #5}
}
\def\doplotaux#1#2#3#4#5{%
    \ifsummary\doplotraux{#1}{#2}{#4_mean}{#4_confmin}{#4_confmax}{#4_min}{#4_max}{#5}{#3}
    \else\addplotnamedtable [x = #3, y = #4]{#1Aux#2};
    \fi%
    \addlegendentry{\csname#1Name\endcsname~ / #5}
}

\def\plots#1{%
    \ifsummary
        \pgfplotstableset{col sep=comma}
        \expandafter\pgfplotsinvokeforeach\alglist{
            \pgfplotstableset{col sep=comma}
            \pgfplotstablereadnamed{\datapath\csname##1Filename\endcsname_log_summary.csv}{##1Performance}
        }
        \expandafter\pgfplotsinvokeforeach\alglist{
            \pgfplotstableset{col sep=comma}
            \pgfplotstablereadnamed{\datapath\csname##1Filename\endcsname_aux_summary.csv}{##1Aux}
        }
        \expandafter\pgfplotsinvokeforeach\alglistAlt{
            \pgfplotstableset{col sep=comma}
            \pgfplotstablereadnamed{\datapathAlt\csname##1Filename\endcsname_log_summary.csv}{##1PerformanceAlt}
        }
        \expandafter\pgfplotsinvokeforeach\alglistAlt{
            \pgfplotstableset{col sep=comma}
            \pgfplotstablereadnamed{\datapathAlt\csname##1Filename\endcsname_aux_summary.csv}{##1AuxAlt}
        }
        \pgfplotstableset{col sep=tab}
    \else
        \expandafter\pgfplotsinvokeforeach\alglist{
            \pgfplotstablereadnamed{\datapath\csname##1Filename\endcsname_log.txt}{##1Performance}
        }
        \expandafter\pgfplotsinvokeforeach\alglist{
            \pgfplotstablereadnamed{\datapath\csname##1Filename\endcsname_aux.csv}{##1Aux}
        }
        \expandafter\pgfplotsinvokeforeach\alglistAlt{
            \pgfplotstablereadnamed{\datapathAlt\csname##1Filename\endcsname_log.txt}{##1PerformanceAlt}
        }
        \expandafter\pgfplotsinvokeforeach\alglistAlt{
            \pgfplotstablereadnamed{\datapathAlt\csname##1Filename\endcsname_aux.csv}{##1AuxAlt}
        }
    \fi
    \pgfplotstablereadnamed{\datapath init_aux.csv}{AuxInit}
    \pgfplotstablegetelem{0}{k_true}\of{\AuxInit}\let\ktrue\pgfplotsretval
    \pgfplotstablegetelem{0}{c1_true}\of{\AuxInit}\let\conetrue\pgfplotsretval
    \pgfplotstablegetelem{0}{c2_true}\of{\AuxInit}\let\ctwotrue\pgfplotsretval
    \pgfplotstablegetelem{0}{k_noisy}\of{\AuxInit}\let\knoisy\pgfplotsretval
    \pgfplotstablegetelem{0}{c1_noisy}\of{\AuxInit}\let\conenoisy\pgfplotsretval
    \pgfplotstablegetelem{0}{c2_noisy}\of{\AuxInit}\let\ctwonoisy\pgfplotsretval
    \pgfplotstablereadnamed{\datapathAlt init_aux.csv}{AuxInitAlt}
    \pgfplotstablegetelem{0}{k_true}\of{\AuxInitAlt}\let\ktrueAlt\pgfplotsretval
    \pgfplotstablegetelem{0}{c1_true}\of{\AuxInitAlt}\let\conetrueAlt\pgfplotsretval
    \pgfplotstablegetelem{0}{c2_true}\of{\AuxInitAlt}\let\ctwotrueAlt\pgfplotsretval
    \pgfplotstablegetelem{0}{k_noisy}\of{\AuxInitAlt}\let\knoisyAlt\pgfplotsretval
    \pgfplotstablegetelem{0}{c1_noisy}\of{\AuxInitAlt}\let\conenoisyAlt\pgfplotsretval
    \pgfplotstablegetelem{0}{c2_noisy}\of{\AuxInitAlt}\let\ctwonoisyAlt\pgfplotsretval
    \begin{tikzpicture}[baseline]
        \begin{axis}[%
                xmode = log,
                ymode = log,
                xmin = 1,
                xmax = \maxiter,
                width = \figwidth,
                height = \figheight,
                ylabel = {Value},
                xlabel = {Iteration},
                performanceplot,
                valueplot,
                legend style = {
                        xshift = 2ex,
                    },
                legend pos = north east,
            ]

            \expandafter\pgfplotsinvokeforeach\alglist {
                \doplot{##1}{}{iter}{value}{\principalname}
            }
            \expandafter\pgfplotsinvokeforeach\alglistAlt {
                \doplot{##1}{Alt}{iter}{value}{\altname}
            }
        \end{axis}
    \end{tikzpicture}%
    \hfill%
    \ifsummary\begin{minipage}[b]{\figwidth}\begingroup
        \pgfplotsset{
            xticklabel style={font=\nullfont, inner sep=0pt, outer sep=0pt},
            xlabel style={font=\nullfont, inner sep=0pt, outer sep=0pt}
            xlabel=\empty,
            xticklabel=\empty,
        }%
    \else%
        \begin{tikzpicture}[baseline]%
            \begin{axis}[%
                    width = \figwidth,
                    height = \figheight,
                    xmode = log,
                    ymode = log,
                    ylabel = {Value},
                    xlabel = {CPU time (seconds)},
                    performanceplot,
                    valueplot,
                    ignore legend
                ]

                \expandafter\pgfplotsinvokeforeach\alglist {
                    \doplot{##1}{}{cpu_time}{value}{\principalname}
                }
                \expandafter\pgfplotsinvokeforeach\alglistAlt {
                    \doplot{##1}{Alt}{cpu_time}{value}{\altname}
                }
            \end{axis}
        \end{tikzpicture}%
        \pgfplotsset{xlabel = {Iteration}}%
        \\%
    \fi%
    \begin{tikzpicture}[baseline]
        \begin{axis}[%
                width = \figwidth,
                height = \twofigheight,
                ylabel = {Discrepancy},
                xmode = log,
                xmin = 1,
                xmax = \maxiter,
                ymode = log,
                performanceplot,
                valueplot,
                ignore legend,
            ]

            \expandafter\pgfplotsinvokeforeach\alglist {
                \doplot{##1}{}{iter}{joint_discrepancy}{\principalname}
            }
            \expandafter\pgfplotsinvokeforeach\alglistAlt {
                \doplot{##1}{Alt}{iter}{joint_discrepancy}{\altname}
            }
        \end{axis}
    \end{tikzpicture}%
    \ifsummary\endgroup\\\fi%
    \begin{tikzpicture}[baseline]
        \begin{axis}[%
                width = \figwidth,
                height = \twofigheight,
                xmode = log,%
                xmin = 1,
                xmax = \maxiter,
                ymin = 0,
                ylabel = {Spike count},
                xlabel = {Iteration},
                const plot,
                performanceplot,
                ignore legend,
            ]

            \expandafter\pgfplotsinvokeforeach\alglist {
                \doplot{##1}{}{iter}{n_spikes}{\principalname}
            }
            \expandafter\pgfplotsinvokeforeach\alglistAlt {
                \doplot{##1}{Alt}{iter}{n_spikes}{\altname}
            }
        \end{axis}
    \end{tikzpicture}%
    \ifsummary\end{minipage}\!\!\!\\\fi%
    \\%
    \begin{tikzpicture}[baseline]
        \begin{axis}[%
                width = \figwidth,,
                height = \twofigheight,
                xmode = log,
                xmin = 1,
                xmax = \maxiter,
                ymode = normal,
                ylabel = {$k$},
                performanceplot,
                valueplot,
                ignore legend,
            ]

            \expandafter\pgfplotsinvokeforeach\alglist {
                \doplotaux{##1}{}{iter}{k}{\principalname}
            }
            \addplot[color = black, line width=0.7pt] coordinates {(1, \ktrue) (\maxiter, \ktrue) };
            \ifsummary\else
                \addplot[color=black, dotted, line width=0.7pt] coordinates {(1, \knoisy) (\maxiter, \knoisy) };
            \fi
        \end{axis}
    \end{tikzpicture}%
    \hfill%
    \begin{tikzpicture}[baseline]
        \begin{axis}[%
                width = \figwidth,,
                height = \twofigheight,
                xmode = log,
                xmin = 1,
                xmax = \maxiter,
                ymode = normal,
                ylabel = {$k$},
                performanceplot,
                valueplot,
                ignore legend,
                cycle list shift=3,
            ]

            \expandafter\pgfplotsinvokeforeach\alglist {
                \doplotaux{##1}{Alt}{iter}{k}{\altname}
            }
            \addplot[color = black, line width=0.7pt] coordinates {(1, \ktrueAlt) (\maxiter, \ktrueAlt) };
            \ifsummary\else
                \addplot[color=black, dotted, line width=0.7pt] coordinates {(1, \knoisyAlt) (\maxiter, \knoisyAlt) };
            \fi
        \end{axis}
    \end{tikzpicture}%
    \\%
    \begin{tikzpicture}[baseline]
        \pgfplotsset{set layers}
        \begin{axis}[%
                width = \figwidth,
                height = \twofigheight,
                xmode = log,
                xmin = 1,
                xmax = \maxiter,
                ymode = normal,
                ylabel = {$c_1$},
                performanceplot,
                valueplot,
                ignore legend,
            ]

            \expandafter\pgfplotsinvokeforeach\alglistAlt {
                \doplotaux{##1}{}{iter}{c1}{\principalname}
            }
            \addplot[color = black, line width=0.7pt] coordinates {(1, \conetrue) (\maxiter, \conetrue) };
            \ifsummary\else
                \addplot[color=black, dotted, line width=0.7pt] coordinates {(1, \conenoisy) (\maxiter, \conenoisy) };
            \fi
        \end{axis}
    \end{tikzpicture}%
    \hfill%
    \begin{tikzpicture}[baseline]
        \pgfplotsset{set layers}
        \begin{axis}[%
                width = \figwidth,
                height = \twofigheight,
                xmode = log,
                xmin = 1,
                xmax = \maxiter,
                ymode = normal,
                ylabel = {$c_1$},
                performanceplot,
                valueplot,
                ignore legend,
                cycle list shift=3,
            ]

            \expandafter\pgfplotsinvokeforeach\alglistAlt {
                \doplotaux{##1}{Alt}{iter}{c1}{\altname}
            }
            \addplot[color = black, line width=0.7pt] coordinates {(1, \conetrueAlt) (\maxiter, \conetrueAlt) };
            \ifsummary\else
                \addplot[color=black, dotted, line width=0.7pt] coordinates {(1, \conenoisyAlt) (\maxiter, \conenoisyAlt) };
            \fi
        \end{axis}
    \end{tikzpicture}%
    \\%
    \begin{tikzpicture}[baseline]
        \begin{axis}[%
                width = \figwidth,
                height = \twofigheight,
                xmode = log,
                xmin = 1,
                xmax = \maxiter,
                ymode = normal,
                ylabel = {$c_2$},
                xlabel = {Iteration},
                performanceplot,
                valueplot,
                ignore legend,
            ]

            \expandafter\pgfplotsinvokeforeach\alglistAlt {
                \doplotaux{##1}{}{iter}{c2}{\principalname}
            }

            \addplot[color = black, line width=0.7pt] coordinates {(1, \ctwotrue) (\maxiter, \ctwotrue) };
            \ifsummary\else
                \addplot[color=black, dotted, line width=0.7pt] coordinates {(1, \ctwonoisy) (\maxiter, \ctwonoisy) };
            \fi
        \end{axis}
    \end{tikzpicture}%
    \hfill%
    \begin{tikzpicture}[baseline]
        \begin{axis}[%
                width = \figwidth,
                height = \twofigheight,
                xmode = log,
                xmin = 1,
                xmax = \maxiter,
                ymode = normal,
                ylabel = {$c_2$},
                xlabel = {Iteration},
                performanceplot,
                valueplot,
                ignore legend,
                cycle list shift=3,
            ]

            \expandafter\pgfplotsinvokeforeach\alglistAlt {
                \doplotaux{##1}{Alt}{iter}{c2}{\altname}
            }

            \addplot[color = black, line width=0.7pt] coordinates {(1, \ctwotrueAlt) (\maxiter, \ctwotrueAlt) };
            \ifsummary\else
                \addplot[color=black, dotted, line width=0.7pt] coordinates {(1, \ctwonoisyAlt) (\maxiter, \ctwonoisyAlt) };
            \fi
        \end{axis}
    \end{tikzpicture}%
}

\newif\ifplots
\plotstrue

\begin{figure}[t!]
    \def\alglist{{radonFB, sFBradon}}
    \let\alglistAlt\alglist
    \def\datapath{aux_}
    \def\principalname{experiment \#1}
    \def\datapathAlt{aux2_}
    \def\altname{experiment \#2}
    \summaryfalse\plots{#1}%
    \caption{%
        Algorithm performance. Top row: Function value in terms of iteration count (left) and CPU time (right).
        Second row: source count in $\mu^k$ (left) and discrepancy \eqref{eq:discrepancy}.
        Remaining rows: evolution of auxiliary variables (left: experiment \#1, right: experiment \#2). The solid horizontal line indicates the ground truth, the dotted horizontal line the measurement.
    }
    \label{fig:performance}
\end{figure}

The reconstructed and true concentration $u$ and measure $\mu$ for experiments \#1 and \#2 are in \cref{fig:recos1,fig:recos2}, respectively.
The performance of both optimisation  algorithms (basic and sliding) is graphed in \cref{fig:performance} for both experiments.
The graphs indicate the function value versus iteration number and CPU time, as well as, with respect to the iteration number, the number of spikes (Dirac deltas) in $\mu^k$, the evolution of the auxiliary variables, and the discrepancy
\begin{equation}
    \label{eq:discrepancy}
    \sqrt{\inf_{\omega \in \subdiff_\mu[F+R](x^{n+1})}  \norm{\omega}_\infty^2 +  \tau (2.0 / \sigma_p - L_{k,c})\norm{(k^{n+1}-k^n, c^{n+1}-c^n)}^2}.
\end{equation}
Due to the implicit form of the update \eqref{eq:alg-sliding:kc-step}, this discrepancy is simply a norm on a minimal element of the total subdifferential $\subdiff[F+R](x^{n+1})$, with a scaling that corresponds to precise expressions from the convergence proof.
The discrepancy is generally a more precise indicator of convergence than the function value, which may be very small far away from a solution.
It also corresponds to the convergence claims of \cite[Theorem 7.2]{tuomov-pointsource}.

In both experiments, the sliding method almost perfectly recovers the true sources, hence the true concentration, while the basic method appears to get stuck at a local stationary point.
Although we report the reconstruction after \maxiter iterations to ensure good recovery of the sources, in terms of function value, both algorithms are very close to the final value of the objective function already after 1000 iterations.

Overall, we can say that the sliding method performs well at leak detection with simultaneous refinement of the scalar convection (wind) parameter , and diffusion (affected e.g. by temperature) parameter, bothmeasured with  2\% noise.
It appears that the sliding helps to avoid local stationary points in this nonconvex optimisation problem.

\Cref{fig:performance} suggests a potential \emph{contrast loss} phenomenon for the auxiliary parameters $k$ and $c$: their reconstruction tends to converge below the noisy measurement (dotted horizontal line).
While in the case of experiment \#1 (left) the reconstructions turn out to well match the ground truth (solid horizontal line), in case of experiment \#1 (right), in some cases, the reconstructions overshoot the ground truth, below it.
In our additional experiments, this phenomenon was quite consistent over random realisations, independent of whether the noisy measurement is smaller or larger in magnitude than the ground truth.
While such an effect could be caused by other issues, contrast loss being a well-known feature of image reconstruction with an $L^2$-squared fitting term \cite{duval2009tvls}, suggests that this is a feature that is unavoidable with squared fitting terms for the auxiliary parameters, i.e., with Gaussian noise.

To help bring further clarity on this, we also illustrate in \cref{fig:summary} the sensitivity of the method to perturbations in the regularisation and step length parameters $\alpha$, $\alpha_c$, $\alpha_k$, $\tau$, $\sigma_p$, and $\theta$.
Moreover, each perturbed computation is performed over 8 different realisations of the noise in the problem.
The latter is highly relevant because any single noisy measurement of the scalar quantities $k$, $c_1$, and $c_2$ can introduce bias.
The regularisation parameters are kept the same, halved, or doubled, compared to the basic experimental setup.
The step length parameters, which were manually chosen near a stable maximum, or kept the same, halved, or divided by four. Over the resulting 104 experimental runs, \cref{fig:summary} illustrates the mean value, the minimum and maximum value (as a lightly coloured region), as well as the 95\% confidence interval (as a more darker coloured region, which tends to be almost indistinguishable from the mean line) of the same quantities that were illustrated in \cref{fig:performance} for a single experimental run.
From these figures, we can see that, on average, the auxiliary parameters $k$ and $c$ are also well reconstructed, albeit, on average, with some contrast loss.

\begin{figure}[t!]
    \def\alglist{{radonFB, sFBradon}}
    \let\alglistAlt\alglist
    \def\maxiter{5000}
    \def\datapath{aux_}
    \def\principalname{experiment \#1}
    \def\datapathAlt{aux2_}
    \def\altname{experiment \#2}
    \summarytrue\plots{#1}%
    \caption{%
        Algorithm and problem sensitivity. Lightly shaded areas indicate minimum and maximum over the varying parameters (see \cref{sec:numerical:res}), darkly shaded areas the 95\% confidence interval, and the lines the mean.
        Top row: Function value in terms of iteration count (left) and source count in $\mu^k$ and discrepancy \eqref{eq:discrepancy} (right).
        Remaining rows: evolution of auxiliary variables (left: experiment \#1, right: experiment \#2). The solid horizontal line indicates the ground truth.
    }
    \label{fig:summary}
\end{figure}

\appendix

\end{document}